\documentclass[11]{amsart}

\usepackage{amssymb}
\usepackage{mathtools}
\usepackage{enumerate}
\usepackage[all,cmtip]{xy}
\usepackage{comment}
\usepackage{soul}
\usepackage{tikz-cd}

\usepackage{mathrsfs}
\usepackage{verbatim}

\usepackage{hyperref}
\hypersetup{
    colorlinks=true,
    linkcolor=blue,
    filecolor=magenta,      
    urlcolor=blue,
    citecolor=blue
}

\usepackage[capitalise]{cleveref}
\usepackage{tikz}
\usetikzlibrary{matrix,arrows}

\newtheorem{thm}{Theorem}[section]
\newtheorem{lem}[thm]{Lemma}
\newtheorem{prop}[thm]{Proposition}
\newtheorem{cor}[thm]{Corollary}

\newtheorem{defprop}[thm]{Definition/Proposition}

\theoremstyle{definition}
\newtheorem{defn}[thm]{Definition}
\newtheorem{rmk}[thm]{Remark}

\usepackage{commands}

\newcommand{\leqp}{%
  \mathrel{\raisebox{-0.5ex}{$\scriptscriptstyle($}}%
  \leq
  \mathrel{\raisebox{-0.5ex}{$\scriptscriptstyle)$}}%
}

\numberwithin{equation}{section}

\usepackage{xcolor}

\begin{document}

\title[Projectivity of moduli spaces]{Projectivity of Moduli Spaces in the Higher-Rank DT/PT Correspondence}
\date{\today}
\keywords{moduli spaces of coherent sheaves, PT-stability, Bridgeland stability}
\subjclass[2020]{14D20, 14D23}

\author{Mihai Pavel}
\address{Mihai Pavel\newline
	\indent Simion Stoilow Institute of Mathematics of the Romanian Academy\newline 
\indent 21 Calea Grivitei Street, 010702, Bucharest, Romania}
\email{cpavel@imar.ro}

\author{Tuomas Tajakka}
\address[]{Tuomas Tajakka}
\email{tuomas.tajakka@gmail.com}

\begin{abstract}
\begin{comment}
{\color{red} TO BE REWRITTEN}
We construct a globally generated line bundle on the moduli stack of higher-rank PT-semistable objects over a smooth projective threefold and analyze the extent to which it separates points. Furthermore, when the rank and degree are coprime, we refine our construction to obtain an explicit ample line bundle on the corresponding coarse moduli space of PT-stable objects, thereby establishing its projectivity.

We conclude with an application to Bridgeland stability conditions on projective threefolds of Picard rank 1.
\end{comment}
We study projectivity of moduli spaces on the DT/PT wall crossing in Bridgeland and polynomial stability on a smooth, projective threefold. First, we construct a globally generated line bundle on the moduli stack of higher-rank PT-semistable objects and analyze the extent to which it separates points. Next, when the rank and degree are coprime, we refine our construction to obtain an explicit ample line bundle on the corresponding coarse moduli space of PT-stable objects, thereby establishing its projectivity. 

Finally, we consider certain classes of threefolds for which PT-stability is realized as a Bridgeland stability condition and establish projectivity also on the wall separating the Gieseker and PT chambers.

\end{abstract}

\maketitle

\section{Introduction}\label{sect:Introduction}
The DT/PT correspondence, as formulated by Pandharipande and Thomas in \cite{PT}, establishes a relationship between two curve-counting invariants on a projective Calabi--Yau threefold $X$. On one side are the classical Donaldson--Thomas (DT) invariants constructed by Thomas in \cite{ThomasDT}, which virtually count $1$-dimensional closed subschemes of $X$. The moduli space of these objects is the Hilbert scheme of curves. On the other side are the Pandharipande--Thomas (PT) invariants introduced in \cite{PT}, which virtually count stable pairs on $X$. By definition, a stable pair $(F,s)$ on $X$ is a $2$-term complex of the form
\[
    \Oh_X  \xrightarrow{s} F
\]
where $F$ is a pure sheaf of dimension $1$ and the cokernel of $s$ is $0$-dimensional. Moduli spaces of stable pairs can be constructed as GIT quotients of coherent systems, in the sense of Le Potier \cite{lePotier1993Systems}. They are moreover projective, in parallel with the corresponding moduli spaces in DT theory. 

The objects underlying the two enumerative theories can be described in terms of certain stability conditions on the derived category of coherent sheaves on $X$, which arise as ``large volume limits" of Bridgeland stability conditions. This connection was independently observed by Bayer \cite{bayer-polynomial} and Toda \cite{toda-limitstable}, who interpreted the DT/PT correspondence as a wall-crossing phenomenon within a family of such stability conditions. Moreover, this approach led to the formulation of a higher-rank version of the DT/PT correspondence, which was subsequently proved by Toda \cite{Toda-DT/PT} along with the original (rank-$1$) case. An alternative proof of the rank-$1$ DT/PT correspondence, also based on wall-crossing techniques, was previously obtained by Bridgeland \cite{Bridgeland-DT/PT} (see also \cite{Toda-DT/PT-I, Stoppa-Thomas} for the Euler-characteristic analogue of the correspondence).

The moduli spaces arising in the (higher-rank) DT/PT correspondence can be defined over any polarized, smooth, complex projective threefold $(X,H)$.  On the higher-rank DT side, we have the moduli space of $H$-Gieseker-semistable sheaves on $X$, which can be constructed as a projective GIT quotient \cite{gieseker77,maruyama78moduli}. This aligns with the rank-$1$ case, since the Hilbert scheme of curves can be identified with the Gieseker moduli space of rank-$1$ torsion-free sheaves with trivial determinant.
Our first goal is to establish projectivity on the PT side in higher rank by studying the moduli space of PT-stable objects. We will recall the notion of PT-stability on the derived category of coherent sheaves on $X$ in \cref{sect:PT-semistable objects}, in a framework that recovers the stable pairs of Pandharipande and Thomas as stable objects of rank $1$ and trivial determinant. 

Moduli spaces of higher-rank PT-stable objects provide alternative compactifications of the moduli of $\mu$-stable reflexive sheaves. Indeed, the moduli theory of PT-semistable objects was developed by J. Lo in \cite{lo-PT1, lo-PT2}, who showed that the moduli stack of PT-semistable objects with fixed Chern character is a universally closed algebraic stack of finite type, which moreover admits a proper coarse moduli space in the absence of strictly semistable objects. However, unlike the rank 1 and trivial determinant case 
considered by Pandharipande and Thomas, the moduli problem for higher-rank PT-stable objects is not known to be associated with a GIT problem, and hence it is unknown whether the moduli spaces are projective. In this paper, we answer this question affirmatively in the case of coprime rank and degree, thereby confirming a result expected in the field (see \cite[Remark 3.22]{toda-limitstable}). 

In their recent work, Jardim, Lo, Maciocia, and Martinez \cite{JLMM2025HigherDTPT} studied the wall-crossing transformation between the DT and PT moduli spaces in the framework of Bridgeland stability conditions. Their analysis is carried out for a projective threefold $X$ of Picard rank $1$ admitting Bridgeland stability conditions constructed via the double-tilt procedure of \cite{bayer2014bridgeland}. Such stability conditions are known to exist for threefolds in the following cases: Fano threefolds \cite{macriBGthreefold, Schmidt2023Sheaves, LiBridgeland}, abelian threefolds \cite{MaciociaPiyaratne, BMS-bridgeland}, quintic threefolds \cite{LiQuintic}, threefolds with nef tangent bundles \cite{KosekiNefTangent}, general weighted hypersurfaces in $\P(1, 1, 1, 1, 2)$ or $\P(1, 1, 1, 1, 4)$ \cite{Koseki}, and Calabi-Yau threefolds given as the complete intersection of a quadric and a quartic in $\P^5$ \cite{LiuStability}. 

Let $v \in \Knum(X)$ be a numerical class of positive rank such that $\rk(v)$ and $H^2\ch_1(v)$ are coprime.
In this case, it was shown in \cite{JLMM2025HigherDTPT} that both the Gieseker moduli space $\rM^\mathrm{\mu}(v)$ of $\mu$-stable sheaves of class $v$ and the moduli space $\rM^{\mathrm{PT}}(v)$ of PT-stable objects can be realized as Bridgeland moduli spaces. Moreover, these spaces are separated by a single wall in the space $\Stab(X)$ of Bridgeland stability conditions. They can be connected by a wall-crossing diagram
\[
    \xymatrix{ \rM^\mathrm{\mu}(v) \ar[rd] &  &  \rM^{\mathrm{PT}}(v) \ar[ld]  \\  & \rM_{\overline{\sigma}}(v)  &  }
\]
where $\overline{\sigma} \in \Stab(X)$ is a stability condition on the separating wall, and $\rM_{\overline{\sigma}}(v)$ is the corresponding Bridgeland moduli space of $\overline{\sigma}$-semistable objects of class $v$  (see op. cit. or \cref{sec:Bridgeland} for details). As noted above, it is known that the Gieseker moduli space $\rM^\mathrm{\mu}(v)$ is projective. This paper completes the picture by showing that $\rM^{\mathrm{PT}}(v)$ and $\rM_{\overline{\sigma}}(v)$ are also projective schemes. 

\subsection*{Main results}
We summarize our main results and the ideas underlying them in the following.

\begin{thm}\label{thm:MainProj}
Let $(X,H)$ be a polarized, smooth, projective threefold over $\C$, and let $v \in \Kn(X)$ be a class of positive rank such that $\rk(v)$ and $H^2\ch_1(v)$ are coprime.
\begin{enumerate}[(a)]
    \item The coarse moduli space $\rM^{\mathrm{PT}}_Z(v)$ of PT-stable objects is projective for any PT-stability function $Z$ on $(X,H)$. 
    \item Assume that $X$ is a projective threefold of Picard rank $1$ admitting double-tilt Bridgeland stability conditions. Then the
    Bridgeland moduli space $\rM_{\overline{\sigma}}(v)$ is projective for a stability condition $\overline{\sigma} \in \Stab(X)$ on the wall separating the Gieseker and PT chambers.
\end{enumerate}
\end{thm}

We begin our study with some general results that do not require the coprime assumption. In \cref{sect:moduli spaces}, we introduce a quasi-compact algebraic stack $\cM_X(v)$, which contains both the stack $\cM^{\mathrm{Gss}}_H(v)$ of $H$-Gieseker-semistable sheaves and the stack $\cM^{\mathrm{PT}}_Z(v)$ of PT-semistable objects. We prove in \cref{thm:SemiAmple} the existence of a \textit{semiample} determinantal line bundle $\cL_1$ on $\cM_X(v)$, and then describe in \cref{subsect:SeparationResult} the extent to which its sections separate points. A key ingredient here is provided by the Mehta-Ramanathan type restriction theorems \cite{mehta1982semistable,mehta1984restriction,pavel2022restriction}. This restriction method originates from the works of Le Potier \cite{le1992fibre} and J. Li \cite{li1993algebraic}, who employed it to construct an algebro-geometric analogue of the Donaldson-Uhlenbeck compactification. In our case, we prove that the restriction of a $\C$-point of $\cM_X(v)$ (and in particular of a PT-semistable object) to a general surface $D \in |aH|$ of sufficiently large degree is isomorphic to a Gieseker-semistable sheaf. This, in turn, makes it possible to construct determinantal $\theta$-sections of $\cL_1$ using the results of Álvarez-Cónsul and King \cite[Section 7]{alvarez2007functorial}.

However, the line bundle $\cL_1$ does not separate the $\C$-points of $\cM^{\mathrm{PT}}_Z(v)$ (nor those of $\cM^{\mathrm{Gss}}_H(v)$). We improve its positivity by twisting it with another determinantal line bundle $\cL_0$, in a manner analogous to the construction for the Gieseker moduli space \cite[Section 8]{HL}. Indeed, when the rank and degree are coprime, we show in \cref{thm:MainAmpleness} that a positive linear combination of $\cL_1$ and $\cL_0$ descends to an \textit{ample} line bundle on the coarse moduli space $\rM^{\mathrm{PT}}_Z(v)$ of PT-stable objects, thus establishing part (a).

Our proof of the projectivity of $\rM^\mathrm{PT}_Z(v)$ relies crucially on the coprime assumption, which implies three useful properties. First, it ensures that there are no strictly PT-semistable objects, which in particular guarantees the existence of the proper coarse moduli space $\rM^{\mathrm{PT}}_Z(v)$. Second, it allows for a concrete characterization of PT-stable objects due to J. Lo, see \cref{prop:PT-partial-characterization}, which is helpful for describing certain strata of $\rM^{\mathrm{PT}}_Z(v)$ \cite{Lo2021PTstableQuotients}. Third, under the coprime assumption, Beentjes and Ricolfi constructed a stratification of the moduli space of PT-stable objects by relative Quot schemes, see \cite[Proposition 5.5 and Appendix B]{Beentjes2021DT/PT}. We use this last property to study the positivity of the line bundle $\cL_0$ (see \cref{sect:projPT}).

We address part (b) in \cref{subsect:projBridgeland}. The results of \cite{JLMM2025HigherDTPT} allow us to identify $\cM_X(v)$ with the Bridgeland moduli stack $\cM_{\overline{\sigma}}(v)$ for a stability condition $\overline{\sigma} \in \Stab(X)$ on the wall separating the Gieseker and PT chambers. We then show that $\cL_1$ descends to an ample line bundle on the good moduli space $\rM_{\overline{\sigma}}(v)$. In fact, our proof of the projectivity of $\rM_{\overline{\sigma}}(v)$ does not rely on the coprime assumption, although a wall-crossing diagram as above is currently only known in the coprime case. This result can be seen as a threefold analogue of the main theorem in \cite{tajakka2022uhlenbeck}, which addresses the case of Bridgeland stability conditions on surfaces.

\subsection*{Acknowledgements}
M.P. would like to thank Daniel Greb, Jochen Heinloth, Emanuele Macrì, Cameron Ruether, Saket Shah, Andrei Stoenic\u a and Matei Toma for helpful conversations and insightful suggestions related to the project. M.P. was partially supported by the PNRR grant CF 44/14.11.2022 \textit{Cohomological Hall algebras of smooth surfaces and applications}, and by a grant of the Ministry of Research, Innovation and Digitalization, CNCS-UEFISCDI, project number PN-IV-P2-2.1-TE-2023-2040, within PNCDI IV.

T.T. would like to thank Jarod Alper, Aaron Bertram, Jochen Heinloth, Nicholas Kuhn, Jason Lo, and Siddharth Mathur for insightful conversations. T.T. was partially supported by the Knut and Alice Wallenberg foundation under grant no. KAW 2019.0493, as well as enjoyed the hospitality of Institut Mittag-Leffler during the program Moduli and Algebraic Cycles in Autumn 2021. Some of the material in this article appeared in T.T.'s doctoral dissertation at the University of Washington \cite{tajakka2021thesis}.

%%%%%%%%%%%%%%%%%%%%%%%%%%%%%%%%%%%%%%%%%%%%%%%%%%%%%
%%%%%%%%%%%%%%%%%%%%%%%%%%%%%%%%%%%%%%%%%%%%%%%%%%%%%
%%%%%%%%%%%%%%%%%%%%%%%%%%%%%%%%%%%%%%%%%%%%%%%%%%%%%
\section{Preliminaries}\label{sect:prelim}
%%%%%%%%%%%%%%%%%%%%%%%%%%%%%%%%%%%%%%%%%%%%%%%%%%%%%
%%%%%%%%%%%%%%%%%%%%%%%%%%%%%%%%%%%%%%%%%%%%%%%%%%%%%
%%%%%%%%%%%%%%%%%%%%%%%%%%%%%%%%%%%%%%%%%%%%%%%%%%%%%
Throughout this paper, we fix a smooth, projective, connected threefold $X$ over $\C$, together with a very ample line bundle $\cO_X(1)$ on $X$. 

In this section, we recall a notion of semistability for coherent sheaves on $X$ following \cite[Section 1.6]{HL}. Let $\Coh(X)$ be the abelian category of coherent sheaves on $X$. For an integer $0 \le e \le 2$, we denote by $\Coh_{3,e}(X)$ the quotient category $\Coh(X)/\Coh_{e-1}(X)$, where $\Coh_{e-1}(X)$ is the full subcategory of $\Coh(X)$ whose objects are sheaves of dimension $\le e-1$. 

For every $E \in \Coh(X)$, let
\begin{equation}\label{eq:hilbPol}
    \rP(E,m) = \int_X \ch(E\otimes \cO_X(m))\td_X = \sum_{i = 0}^3 a_i(E) {m+3-i \choose 3-i}
\end{equation} 
be the Hilbert polynomial of $E$ with respect to the polarization $\cO_X(1)$. By linearity, the Hilbert polynomial extends to an additive function
\[
    \rP : \rK(X) \to \Q[m].
\]
For $0 \le e \le 2$, we define the $e$-truncated Hilbert polynomial by
\[
    \rP_e : \rK(X) \to \Q[m],\quad \rP_e(E,m) = \sum_{i = 0}^{3-e} a_i(E) {m+3-i \choose 3-i}.
\]
Moreover, when $E \in \Coh(X)$ is torsion-free, we denote by 
\[
    \rp_e(E,m) = \frac{\rP_e(E,m)}{\rk(E)}
\]
the reduced $e$-truncated Hilbert polynomial of $E$.

\begin{defn}
Let $1 \le d \le 3$ be an integer and set $e = 3-d$. We say that a coherent sheaf $E$ on $X$ is \textbf{$d$-Gieseker-(semi)stable} if it is torsion-free, and for every coherent subsheaf $0 \neq F \subsetneq E$
\[
    \rp_{e}(F,m) \leqp \rp_{e}(E,m)\quad \text{for }m \gg 0.
\]
\end{defn}

One can easily check that $1$-Gieseker-stability is equivalent to the classical notion of $\mu$-stability, and that $3$-Gieseker-stability is the usual notion of Gieseker-stability. We also have
\begin{align*}
 \mu\text{-stable} \Rightarrow 2\text{-Gieseker-stable} \Rightarrow \text{Gieseker-stable} \Rightarrow \\
 \text{Gieseker-semistable} \Rightarrow 2\text{-Gieseker-semistable} \Rightarrow \mu\text{-semistable}.
\end{align*}
Note that for torsion-free coherent sheaves of coprime rank and degree, all these conditions are in fact equivalent.

In general, every $d$-Gieseker-semistable sheaf $E$ admits a Jordan-H\"older filtration
\[
    E_\bullet: 0 = E_0 \subset E_1 \subset \ldots \subset E_m = E
\]
such that each $E_i/E_{i-1}$ is $d$-Gieseker-stable and $\rp_{3-d}(E_i/E_{i-1})=\rp_{3-d}(E)$. We denote by $\gr(E_\bullet) = \bigoplus_i E_i/E_{i-1}$ the associated graded module of $E_\bullet$. The sheaf $\gr(E_\bullet)$ is uniquely determined in $\Coh_{3,3-d}(X)$, cf. \cite[Theorem 1.6.7]{HL}. We say that two $d$-Gieseker-semistable sheaves $E_1$, $E_2$ are \textbf{$S$-equivalent} if their corresponding graded modules are isomorphic in $\Coh_{3,3-d}(X)$ (that is, they are isomorphic outside a closed subset of codimension $d + 1$ in $X$).

\section{Moduli stacks of complexes}\label{sect:moduli spaces}

%%%%%%%%%%%%%%%%%%%%%%%%%%%%%%%%%%%%%%%%%%%%%%%%%%%%%
%%%%%%%%%%%%%%%%%%%%%%%%%%%%%%%%%%%%%%%%%%%%%%%%%%%%%
%%%%%%%%%%%%%%%%%%%%%%%%%%%%%%%%%%%%%%%%%%%%%%%%%%%%%
Let $\cD_{\pug}$ be the stack of \textit{relatively perfect universally gluable} complexes on $X$, introduced by Lieblich in \cite{lie06}. Namely, $\cD_{\pug}$ is the stack over the big \'etale site of $\C$-schemes which to a $\C$-scheme $S$ associates the groupoid of objects $E \in \Db(S \times X)$ such that
\begin{enumerate}
    \item $E$ is relatively perfect over $S$, see \cite[Tag 0DHZ]{stacks-project}, and
    \item $E$ is universally gluable, i.e., for all geometric points $s \in S$, the derived restriction $E_s \coloneqq E|^\LL_{\{s\} \times X} \in \Db(X)$ to the fiber over $s$ satisfies $$\Ext^i_{\Db(X)}(E_{s},E_{s}) = 0 \quad \text{for } i < 0.$$
\end{enumerate}
Lieblich showed that $\cD_{\pug}$ is an algebraic stack locally of finite type over $\C$, containing many open substacks of interest, such as the stack of flat families of sheaves on $X$ and stacks of tilting hearts arising in the context of Bridgeland stability \cite[Appendix A]{ABL13}. He referred to $\cD_{\pug}$ as the ``mother of all moduli spaces (of sheaves)''.

Consider the following full subcategories
\begin{align*}
    \Coh_{\le 1}(X) ={}& \{ T \in \Coh(X) \mid \dim(T) \le 1 \},\\
    \Coh_{\ge 2}(X) ={}& \{ F \in \Coh(X) \mid \dim(F') \ge 2 \text{ for every coherent subsheaf } F' \subs F \},
\end{align*}
which form a torsion pair on $\Coh(X)$. By tilting along this pair we obtain the heart of a bounded $t$-structure on $\Db(X)$, known as the \textit{heart of perverse sheaves},
\[
    \cA^p(X) = \langle  \Coh_{\ge 2}(X), \Coh_{\le 1}(X)[-1] \rangle.
\]
In other words, $\sA^p(X) \subseteq \Db(X)$ is the full subcategory consisting of objects $E \in \Db(X)$ that fit in an exact triangle
\[ F \to E \to T[-1], \]
where $F \in \Coh_{\ge 2}(X)$ and $T \in \Coh_{\le 1}(X)$. 

\begin{rmk}\label{rmk:convention0}
Our definition of the heart $\sA^p(X)$ differs from that in \cite{lo-PT1}, \cite{lo-PT2}, and \cite{bayer-polynomial} by a shift: the nonzero cohomology sheaves are in degrees $0$ and $1$ rather than $-1$ and $0$. The reason for this choice is purely psychological: if $E \in \Coh(X)$ is a torsion-free sheaf, then $E$ itself, rather than $E[1]$, is contained in $\sA^p(X)$.
\end{rmk}

Following the terminology in \cite[Appendix A]{ABL13}, we consider the stack of torsion theories $(\cT, \cF)$, where $\cT$ (resp. $\cF$) is the stack of flat families of coherent sheaves in $\Coh_{\le 1}(X)$ (resp. in $\Coh_{\ge 2}(X)$). Let $\cD_{\cA^p} \subset \cD_{\pug}$ be the substack of tilted hearts with respect to $(\cT,\cF)$, i.e. the substack whose objects over a $\C$-schemes $S$ are $S$-perfect complexes $E \in \Db(S \times X)$ satisfying: for every geometric point $s \in S$, the derived restriction $E_{s} \in \Db(X)$ to the fiber over $s$ lies in the heart $\cA^p(X)$.

As in \cite[Appendix A, Example 1]{ABL13}, one can show that $(\cT,\cF)$ is \textit{open}, and therefore $\cD_{\cA^p}$ is an \textit{open} substack of $\cD_{\pug}$, cf. \cite[Corollary A4]{ABL13}. Consequently, $\cD_{\cA^p}$ is an algebraic stack locally of finite type over $\C$. 

We consider the substack $\cM_X \subset \cD_{\cA^p}$ whose objects over a $\C$-schemes $S$ are $S$-perfect complexes $E \in \Db(S \times X)$ such that for every geometric point $s \in S$, the derived restriction $E_{s}$ to the fiber over $s$ lies in $\cA^p(X)$ and satisfies: 
\begin{enumerate}[(A)]
        \item\label{eq:A} $\sH^0(E_{s})$ is $2$-Gieseker-semistable,
        \item\label{eq:B} $\sH^1(E_{s})$ is 0-dimensional.
        %\item $\Hom_{\Db(X)}(Q[-1],E|^\LL_{s}) = 0$ for every $0$-dimensional sheaf $Q \in \Coh(X)$.
    \end{enumerate}   
Since $E_{s}$ lies in the heart $\cA^p(X)$, we have $\sH^i(E_{s}) = 0$ for all $i \neq 0,1$.

\begin{comment}
\begin{lem}\label{lem:pureQ}
Assume $E \in \cM_X$ is a $\C$-point fitting in an exact triangle
\[ F \to E \to T[-1] \]
where $F \in \Coh(X)$ is torsion-free and $T \in \Coh(X)$ is 0-dimensional. Then the sheaf $Q = F^\dd/F$ is either zero or pure of dimension 1, and $F$ has homological dimension at most 1.
\end{lem}
\begin{proof}
If $Q$ is nonzero and not pure, the maximal 0-dimensional subsheaf $Q_0 \subs Q$ is also nonzero. We have exact sequences
\[ 0 \to Q[-1] \to F \to F^\dd \to 0 \]
and
\[ 0 \to Q_0[-1] \to Q[-1] \to Q/Q_0[-1] \to 0 \]
in $\sA^p(X)$. Thus, we get a nonzero map
\[ Q_0[-1] \to Q[-1] \to F \to E \]
as a sequence of inclusions in the heart of a bounded t-structure of perverse sheaves. But this is impossible since $\Hom_{\Db(X)}(Q_0[-1], E) = 0$ by assumption.

If $Q$ is zero, then $F$ is reflexive and has homological dimension at most 1 by the Auslander-Buchsbaum formula. Otherwise, $Q$ is nonzero, and consider the short exact sequence
\[
    0 \to F \to F^{\vee \vee} \to Q \to 0.
\]
By \cite[Tag 00LX]{stacks-project} we obtain that
\[
    \depth_x(F) \ge \min\, \{ \depth_x(F^{\vee \vee}), \depth_x (Q) + 1\}
\]
for every $x \in X$. As $F^{\vee \vee}$ is reflexive, $\depth_x(F^{\vee \vee}) \ge 2$ whenever $\dim(\cO_{X,x}) \ge 2$. Also, when $\dim(\cO_{X,x}) \ge 3$, we have $\depth_x(Q) \ge 1$ since $Q$ is pure of dimension $1$, from which we further obtain that $\depth_x(F) \ge 2$. By using again the Auslander-Buchsbaum formula we deduce that $F$ has homological dimension at most 1.
\end{proof}
\end{comment}

Let $v \in \Kn(X)$ be a numerical class of positive rank. We denote by $\cM_X(v) \subset \cM_X$ the substack of objects of fixed numerical class $v$. We recall the definition of a bounded family of sheaves, cf. \cite[Definition 1.7.5]{HL}, which will be used in the following result.

\begin{defn}
A family $\frak{F}$ of coherent sheaves on $X$ is called \textbf{bounded} if there exists a scheme $S$ of finite type over $\C$ and a coherent sheaf $\cF \in \Coh(S \times X)$ such that every sheaf $F \in \frak{F}$ is isomorphic to a fiber $\cF_s \coloneqq \cF|_{\{s\} \times X}$ for some closed point $s \in S$.
\end{defn}

\begin{prop}\label{prop:qCompact}
 The stack $\cM_X(v)$ is a quasi-compact, open substack of $\cD_{\cA^p}$.
 \end{prop}
 \begin{proof}
 One sees as in \cite[Proposition 3.1]{lo-PT2} that properties \eqref{eq:A} and \eqref{eq:B} above together form an open condition. Therefore $\cM_X(v)$ is an open substack of $\cD_{\cA^p}$.
 
The quasi-compactness of $\cM_X(v)$ follows from the proof of \cite[Proposition 3.4]{lo-PT1}. We briefly sketch the argument here. Recall that every $\C$-point $E \in \cM_X(v)$ fits in an exact triangle of the form
\[
    \cH^0(E) \to E \to \cH^1(E)[-1]
\]
such that $\cH^0(E) \in \Coh(X)$ is $2$-Gieseker-semistable and $\cH^1(E) \in \Coh(X)$ is $0$-dimensional. By the additivity of the Hilbert polynomial in exact triangles, we obtain
\[
    \rP_{1}(\cH^0(E)) = \rP_{1}(v), \quad a_3(\cH^0(E)) = a_3(v) + a_3(\cH^1(E)),
\]
where $a_3(-)$ denotes the constant term of the Hilbert polynomial, as in \eqref{eq:hilbPol}. Since $\cH^1(E)$ is $0$-dimensional, we also have
\[
    a_3(\cH^1(E)) = \length(\cH^1(E)) \ge 0,
\]
and thus $a_3(\cH^0(E)) \ge a_3(v)$. Therefore, the sheaves $\cH^0(E)$ are contained in the family of $\mu$-semistable sheaves $F \in \Coh(X)$ satisfying $a_i(F) = a_i(v)$ for $i = 0,1,2$, and $a_3(F) \ge a_3(v)$, which is bounded by \cite[Theorem 4.8]{maruyama81on}. In particular we see that $a_3(\cH^0(E))$ can take only finitely many values, which shows that the length of $\cH^1(E)$ is uniformly bounded by some integer $N$ independent of $E \in \cM_X(v)$. Since the moduli space of $0$-dimensional sheaves of length $\le N$ on $X$ is quasi-compact, we obtain the boundedness of the family
\[
     \{ \, \cH^1(E) \mid E \in \cM_X(v)(\C) \, \}. 
\]
Hence, $\cM_X(v)$ is quasi-compact, since all its $\C$-points are extensions of coherent sheaves belonging to two bounded families.
 \end{proof}

%%%%%%%%%%%%%%%%%%%%%%%%%%%%%%%%%%%%%%%%%%%%%%%%%%%%%
%%%%%%%%%%%%%%%%%%%%%%%%%%%%%%%%%%%%%%%%%%%%%%%%%%%%%
%%%%%%%%%%%%%%%%%%%%%%%%%%%%%%%%%%%%%%%%%%%%%%%%%%%%%
\section{Determinantal line bundles}

%%%%%%%%%%%%%%%%%%%%%%%%%%%%%%%%%%%%%%%%%%%%%%%%%%%%%
%%%%%%%%%%%%%%%%%%%%%%%%%%%%%%%%%%%%%%%%%%%%%%%%%%%%%
%%%%%%%%%%%%%%%%%%%%%%%%%%%%%%%%%%%%%%%%%%%%%%%%%%%%%

In this section we review the construction and properties of determinantal line bundles on algebraic stacks. Additionally, we construct a \textit{semiample} determinantal line bundle on the stack $\cM_X(v)$ introduced in the previous section.

\subsection{Preliminaries on determinantal line bundles}\label{subsect:detLB}

Let $S$ be an algebraic stack of finite type over $\C$, and let $\sE \in \Db(S \times X)$ be a complex relatively perfect over $S$. Consider the diagram
\begin{center}
    \begin{tikzpicture}
    \matrix (m) [matrix of math nodes, row sep=1em, column sep=1em]
    { & S \times X & \\
    S & & X \\};
    \path[->] 
    (m-1-2) edge node[auto,swap] {$ p $} (m-2-1)
    (m-1-2) edge node[auto] {$ q $} (m-2-3)
    ;
    \end{tikzpicture}
\end{center}
We obtain a group homomorphism
\[ \la_\sE: \rK(X) \to \Pic(S), \]
called the \textbf{Donaldson morphism}, defined by sending a vector bundle $F$ on $X$ to the line bundle
\[ \la_\sE(F) = \det( \rR p_* (\sE \otimes q^*F)) \]
and extending linearly to $\rK(X)$. Moreover, if the complex $\rR p_*(\sE \otimes q^* F)$ can be locally on $S$ expressed as a 2-term complex of locally free sheaves
\[ \cdots \to 0 \to \sG_0 \xrightarrow{f} \sG_1 \to 0 \to \cdots \]
with $\rk(\sG_0) = \rk(\sG_1)$, then the local sections $f: \Oh_S \to \det(\sG_1) \otimes \det(\sG_0)^\vee$ glue to a global section of $\la_\sE(F)^\vee$.
We recall here the following straightforward application of the cohomology and base change theorem \cite[Tag 0A1K]{stacks-project}, which gives a criterion for the existence and non-vanishing of such a section. The proof of \cite[Lemma 4.1]{tajakka2022uhlenbeck} applies verbatim.
\begin{lem}\label{detsection}
    Let $X$ be a smooth, projective variety and $S$ an algebraic stack of finite type over $\C$. Let $\sE \in \Db(S \times X)$ be an $S$-perfect family of objects of class $v \in \Kn(X)$, and let $K \in \Db(X)$.
    \begin{enumerate}[(a)]
        \item If for all $\C$-points $s \in S$, we have $\Hh^i(X, \sE_s \otimes^\L K) = 0$ whenever $i \neq 0, 1$, and 
        \[ \chi(X, \sE_s \otimes^\L K) = \dim \Hh^0(X, \sE_s \otimes^\L K) - \dim \Hh^1(X, \sE_s \otimes^\L K) = 0, \]
        then the line bundle $\la_\sE([K])^\vee$ on $S$ has a canonical section $\de_{K}$.
        \item In addition, if for some $s_0 \in S$ we have 
        \[ \Hh^0(X, \sE_{s_0} \otimes^\L K) = \Hh^1(X, \sE_{s_0} \otimes^\L K) = 0, \]
        then the section $\de_{K}$ is nonzero at $s_0$.
    \end{enumerate}
\end{lem}
 
\subsection{A semiample determinantal line bundle}\label{section:determinantal-bundles-on-PT}

We now specialize to the stack $\cM_X(v)$. Let $\sE$ be the universal complex on $\cM_X(v) \times X$, so that we have the diagram
\begin{center}
    \begin{tikzpicture}
    \matrix (m) [matrix of math nodes, row sep=1em, column sep=1em]
    { & \sE &  \\
    & \cM_X(v) \times X & \\
    \cM_X(v) & & X \\};
    \path[dotted]
    (m-1-2) edge node[auto,swap] {$ $} (m-2-2)
    ;
    \path[->] 
    (m-2-2) edge node[auto,swap] {$ p $} (m-3-1)
    (m-2-2) edge node[auto] {$ q $} (m-3-3)
    ;
    \end{tikzpicture}
\end{center}

Consider the class
\begin{align}\label{eq:grothClass}
     w_{m,n,a} = - \chi(v(m) \cdot [\cO_D])[\cO_X(n)] + \chi(v(n) \cdot [\cO_D])[\cO_X(m)] \in \rK(X)
\end{align}
where $m, n, a > 0$ are integers, and $D \subset X$ is a divisor in $|\cO_X(a)|$. We define the line bundle
\[
    \cL_{m,n,a} \coloneqq \lambda_\cE(w_{m,n,a} \cdot [\cO_D]) \in \Pic(\cM_X(v))
\]
using the Donaldson morphism. Note that the class
\[  
    [\cO_D] = [\cO_X] - [\cO_X(-a)] \in \rK(X)
\]
does not depend on the chosen divisor $D$ in $|\cO_X(a)|$.

We next aim to show the following theorem:

\begin{thm}\label{thm:SemiAmple}
There are integers $m,n,a > 0$ such that the line bundle $\cL_{m,n,a}$ is semiample over $\cM_X(v)$.
\end{thm}

For this, we need some auxiliary results. The first one describes how the $\C$-points of $\cM_X(v)$ behave when restricted to surfaces of $X$.

\begin{lem}\label{lem:DerivedRestriction}
Let $D \subset X$ be a smooth surface of $X$, and let $E \in \cM_X(v)$ be a $\C$-point fitting in an exact triangle
\[
    F \to E \to T[-1]
\]
where $F \in \Coh(X)$ is $2$-Gieseker-semistable and $T \in \Coh(X)$ is $0$-dimensional. Then the derived restriction $E|_D^{\L}$ fits in an exact triangle
\[
    \cH^0(E|_D^{\L}) \to E|_D^{\L} \to \cH^1(E|_D^{\L})[-1]
\]
in $\Db(D)$, and $\cH^1(E|_D^{\L})$ is isomorphic to $T|_D$. 

If, moreover, $D$ does not meet the support of $T$, then $E|_D^{\L}$ is quasi-isomorphic to the restricted sheaf $F|_D$.
\end{lem}
\begin{proof}
By using the short exact sequence
\[
    0 \to F(-D) \to F \to F|_D \to 0
\]
one computes that $\cH^{0}(F|_D^{\L}) \cong F|_D$ and $\cH^{i}(F|_D^{\L}) = 0$ for $i \neq 0$. Similarly, one shows that $\cH^{i}(T|_D^{\L}) = 0$ for $i \neq -1,0$. Then the exact triangle
\[
    F|_D^{\L} \to E|_D^{\L} \to T|_D^{\L}[-1]
\]
leads to the following long exact sequence in sheaf cohomology
\[
   0 \to F|_D \to \cH^0(E|_D^{\L}) \to \cH^{-1}(T|_D^{\L}) \to 0 \to \cH^1(E|_D^{\L}) \to \cH^0(T|_D^{\L}) \to 0.
\]
We see that $\cH^i(E|_D^{\L}) = 0$ for $i \neq 0,1$, and so $E|_D^{\L}$ is supported only in degrees $0$ and $1$. Moreover, 
\[
    \cH^1(E|_D^{\L}) \cong  \cH^0(T|_D^{\L}) \cong \Coker(T(-D) \to T),
\]
hence $\cH^1(E|_D^{\L})$ is isomorphic to $T|_D$. In particular $\cH^1(E|_D^{\L}) = 0$ and $\cH^{-1}(T|_D^{\L}) = 0$ when $D$ avoids the support of $T$.
\end{proof}

\begin{lem}\label{lem:effRestriction}
There exists an integer $a > 0$ such that for every $\C$-point $E \in \cM_X(v)$, the restriction $E|_D^{\L}$ to a general surface $D \in |\cO_X(a)|$ is quasi-isomorphic to a Gieseker-semistable sheaf on $D$. Whenever $\cH^0(E)$ is Gieseker-stable, we can moreover ensure that the restriction $E|_D^{\L}$ is Gieseker-stable.
\end{lem}
\begin{proof}
Recall that for every $\C$-point $E \in \cM_X(v)$, the sheaf $\cH^0(E)$ is 2-Gieseker-semistable (or 2-Gieseker-stable). By the Mehta-Ramanthan type restriction theorems in \cite{pavel2022restriction}, there exists an integer $a(E) > 0$ such that for every $a \ge a(E)$, the restriction $E|_D^{\L} = \cH^0(E)|_D$ to a general divisor $D \in |\cO_X(a)|$ is a Gieseker-semistable (resp. Gieseker-stable) sheaf. A priori, the lower bound $a(E)$ on the degree of $D$ depends on the object $E$; however, we will show below that we can choose a uniform bound $a_0$ independently of $E$. 

Since the family
\[
    \frak{F} = \{\, \cH^0(E) \mid E \in \cM_X(v)(\C) \,\}
\]
is bounded (because $\cM_X(v)$ is quasi-compact), the sheaves $\cH^0(E)$ can be recovered as fibers of a coherent sheaf $\cF \in \Coh(S \times X)$ over a scheme $S$ of finite type over $\C$. Using a flattening stratification of $S$ and the fact that the notion of $2$-Gieseker-semistability is open in flat families, cf.~\cite[Lemma 2.2]{pavel2022restriction}, we may assume that $\cF$ is an $S$-flat family of $2$-Gieseker-semistable sheaves containing all sheaves $\cH^0(E) \in \frak{F}$. Then, for every closed point $s \in S$, there exists an integer $a(s) > 0$ such that for every $a \ge a(s)$ and general divisor $D \in |\cO_X(a)|$, there exists a nonempty open neighborhood $U_a \subset S$ of $s$ with the property that the restriction $$\cF|_{U_a \times D} \in \Coh(U_a \times D)$$ is an $U_a$-flat family of Gieseker-semistable sheaves on $D$. Since $S$ is quasi-compact, we can cover $S$ by finitely many open subsets $U_{a_1},\ldots,U_{a_r}$ of this type. Then any $a \ge a_0 = \max\, \{ a_i \mid i  = 1,\ldots, r\}$ satisfies the required condition.
\end{proof}

\begin{rmk}
When the rank and degree of $v$ are coprime, one can use effective restriction theorems, as established in \cite{flenner1984restrictions,langer2004semistable}, to obtain an explicit lower bound for $a$ in the previous result.
\end{rmk}

\begin{lem}\label{lem:n-regular}
For any integer $a > 0$, there exists an integer $n > 0$ such that for every $\C$-point $E \in \cM_X(v)$ and for every divisor $D \in |\cO_X(a)|$, the sheaf restriction $\cH^0(E)|_D$ is $n$-regular. 
\end{lem}
\begin{proof}
Since the family of coherent sheaves
\[
    \{\, \cH^0(E) \mid E \in \cM_X(v)(\C) \,\}
\]
is bounded, the same is true for the family of restricted sheaves
\[
     \{\, \cH^0(E)|_D \mid E \in \cM_X(v)(\C), D \in |\cO_X(a)| \,\}.
\]
Therefore there exists a sufficiently large $n > 0$ such that all these restrictions $\cH^0(E)|_D$ are $n$-regular, cf. \cite[Lemma 1.7.6]{HL}.    
\end{proof}

Another key ingredient is the following characterization of Gieseker-semistability on smooth surfaces $D \subset X$, which was shown in \cite[Theorem 7.2]{alvarez2007functorial}. Here, Gieseker-semistability on $D$ is defined with respect to the induced polarization $\cO_D(1) \coloneqq \cO_X(1)|_D$.

\begin{thm}\label{thm:Gss}
Let $P \in \Q[T]$ be a fixed polynomial of degree 2, and let $m > n > 0$ be two integers satisfying conditions \cite[(C:1)--(C:5)]{alvarez2007functorial} with respect to $P$. Then for any smooth, integral surface $D \subset X$, any $n$-regular pure sheaf $E$ of Hilbert polynomial $P$ on $(D,\cO_D(1))$ is Gieseker-semistable if and only if there is a map
\[ 
    U_1 \otimes \cO_D(-m) \xrightarrow{\theta} U_0 \otimes \cO_D(-n), 
\]
where $U_1$ and $U_0$ are non-zero vector spaces, such that the linear map 
\begin{align*}
	\Hom(\theta,E) : \Hom(U_0,\rH^0(E(n))) \to \Hom(U_1,\rH^0(E(m))) 
\end{align*}
is invertible, i.e. $\delta_\theta(E) \coloneqq \det \Hom(\theta,E) \neq 0$. 
\end{thm}

\begin{rmk}\label{rmk:criteriaGss}
a) Note that \cite[Theorem 7.2]{alvarez2007functorial} is more general, as it applies to any $n$-regular pure sheaf $E$ of Hilbert polynomial $P$ on $X$, not only to those with smooth integral support. We refer to \cite[p. 23]{alvarez2007functorial} for the precise formulation of conditions (C:1)--(C:5).

b) Once a map $\theta \in \Hom_D(U_1 \otimes \cO_D(-m), U_0 \otimes \cO_D(-n))$ as in the statement exists, then there is a non-empty Zariski-open neighborhood $V$ of $\theta$ such that $\Hom(\theta',E)$ is invertible for every $\theta' \in V$ (since such a property is open).

c) If we consider the complex
\[ 
    K: K_1 = U_1 \otimes \cO_D(-m) \xrightarrow{\theta} U_0 \otimes \cO_D(-n) = K_0, 
\]
then the condition that $\Hom(\theta,E)$ is invertible is equivalent to $\rR\Gamma(E \otimes^{\L} K^\vee) = 0$ whenever $E$ is $n$-regular (see \cite[Remark 7.4]{alvarez2007functorial}). To avoid any confusion, in the following we will view $K^\vee = [K_0^\vee \to K_1^\vee]$ as a complex with $K_0^\vee$ in degree $0$ and $K_1^\vee$ in degree $1$. 
\end{rmk}

\begin{proof}[Proof of \cref{thm:SemiAmple}]
For every $\C$-point $t \in \cM_X(v)$, we denote
\[
    E_t = \cE|^\LL_{\{t\} \times X}, \quad  F_t = \cH^0(E_t), \quad T_t = \cH^1(E_t).
\]

We choose integers $a, n, m > 0$ as follows:
\begin{enumerate}[(a)]
    \item\label{cond:a} First, we choose $a >0$ such that for every $\C$-point $t \in \cM_X(v)$, the restriction ${E_t}|_D^{\L}$ to a general smooth surface $D \in |\cO_X(a)|$ is quasi-isomorphic to a Gieseker-semistable sheaf on $D$.
    \item\label{cond:b} Second, we choose $n > 0$ satisfying conditions (C:1) and (C:2) from \cite{alvarez2007functorial} with respect to the Hilbert polynomial $P$ of $v - v(-a) \in \Knum(X)$, and such that for every $\C$-point $t \in \cM_X(v)$, the sheaf restriction $\cH^0(E_t)|_D$ to any surface $D \in |\cO_X(a)|$ is $n$-regular.
    \item\label{cond:c} Third, we choose $m > n$ such that conditions (C:3)--(C:5) from \cite{alvarez2007functorial} are satisfied (with respect to $P$ and $n$), and such that $2P(n) < P(m)$.
\end{enumerate}
Condition (a) and the second part of (b) can be ensured by \cref{lem:effRestriction} and \cref{lem:n-regular}, respectively. We also note that conditions (C:1)--(C:5) are always fulfilled for sufficiently large choices of $m > n > 0$, cf. \cite[Section 5.1]{alvarez2007functorial}. These conditions will allow us to apply \cref{thm:Gss}. 

Now pick a $\C$-point $t_0 \in \cM_X(v)$, which corresponds by definition to an object $E_{t_0} \in \Db(X)$ fitting in an exact triangle
\[
    F_{t_0} \to E_{t_0} \to T_{t_0}[-1]
\]
where $F_{t_0}$ is a $2$-Gieseker-semistable sheaf, and $T_{t_0}$ is a $0$-dimensional sheaf. Then for a general smooth surface $D \in |\cO_X(a)|$, the restriction $E_{t_0}|^\LL_D \cong F_{t_0}|_D$ is an $n$-regular Gieseker-semistable sheaf on $D$. Therefore, by \cref{thm:Gss}, there is a general map 
\[ 
    K: K_1 = U_1 \otimes \cO_D(-m) \xrightarrow{\theta} U_0 \otimes \cO_D(-n) = K_0
\]
such that $\H^i(E_{t_0}|_D^{\L} \otimes^{\L} K^\vee) = 0$ for all $i$. Here $U_1$ and $U_0$ are (non-zero) vector spaces satisfying
\[
    P(n)\dim(U_0) = P(m)\dim(U_1).
\]
Using the assumption $2P(n) < P(m)$ from \eqref{cond:c}, one easily obtains that $\dim(U_0) > \dim(U_1) + 1$.

Now consider the vector bundle $$K_0 \otimes K_1^\vee = (U_0 \otimes U_1^\vee) \otimes \cO_D(m-n),$$ which is globally generated over $D$, since so is $\cO_X(m-n)$. In this case, the following statement is well-known (see \cite[Section 4]{banica1991Smooth} for a proof): if $u \in \Hom(K^\vee_0, K^\vee_1)$ is a general map, then its $k$-th degeneracy locus $D_k(u) \subseteq D$, supported on the set $\{ x \in D \mid \rank(u(x)) \le k\}$, is either empty or is of pure codimension $(\rank(K^\vee_0)-k)(\rank(K^\vee_1)-k)$ in $D$. Letting $k = \rank(K^\vee_1) - 1$ and using the inequality $\dim(U_0) > \dim(U_1) + 1$, we deduce that $D_k(u)$ is empty. In other words, a general map
\[
    U_0^\vee \otimes \cO_D(n) \to U_1^\vee \otimes \cO_D(m)
\]
is surjective over $D$.

Therefore we may assume that $\theta^\vee$ is surjective over $D$. By using \cref{lem:sectionOnSurface} below, we obtain that $\H^i(E_t|_D^{\L} \otimes^{\L} K^\vee) = 0$ whenever $i \neq 0,1$, and for every $\C$-point $t \in \cM_X(v)$. This further yields by \cref{detsection} a global section $\delta_K$ of $$\lambda_{\cE_D}([K^\vee])^\vee \cong \lambda_{\cE_D}([K])$$ which is non-vanishing at $t_0$, where $\cE_D$ denotes the restriction of the universal complex $\cE$ to $\cM_X(v) \times D$. 

Next we relate $\lambda_{\cE_D}([K])$ and $\cL_{m,n,a}$. 
\begin{comment}
 Since $\H^i(E_t|_D^{\L} \otimes^{\L} K^\vee) = 0$ for $i \neq 0,1$ and for all $t \in \cM_X(v)$, we get by cohomology and base change (see \cite[Tag 0A1K]{stacks-project}) that $\rR^i p_*(\cE_D \otimes^{\L} p^* K^\vee) = 0$ for $i \neq 0, 1$, where $p: \cM_X(v) \times D \to \cM_X(v)$ is the first projection. Therefore   
\end{comment}
Note that
\begin{align*}
    [K] = [K_0] - [K_1] \in \rK(X).
\end{align*}
Letting $d_1 = \dim(U_1)$, $d_0 = \dim(U_0)$, and using $d_1 P(m) = d_0 P(n)$, we get
\begin{align*}
P(m)([K_0] - [K_1])= d_0 (P(n) [\cO_D(m)] - P(m)[\cO_D(n)]) = d_0 w_{m,n,a}|_D
\end{align*}
in $\rK(D)$. Hence
\begin{align}\label{eq:isoK}
	\cL_{m,n,a}^{d_0} = \lambda_{\cE}(w_{m,n,a} \cdot [\cO_D])^{d_0} \cong \lambda_{\cE_D}({w_{m,n,a}}|_D)^{d_0} \cong \lambda_{\cE_D}([K_0] - [K_1])^{P(m)}.
\end{align}
Therefore a power of $\cL_{m,n,a}$ is globally generated at $t_0 \in \cM_X(v)$. Since $\cM_X(v)$ is quasi-compact, we deduce that $\cL_{m,n,a}$ is semiample over $\cM_X(v)$.
\end{proof}

\begin{lem}\label{lem:sectionOnSurface}
Let $m,n,a > 0$ be integers as above, satisfying conditions \eqref{cond:a}--\eqref{cond:c}. Let $E \in \cM_X(v)$ be a $\C$-point fitting in
\[ F \to E \to T[-1], \]
let $D \in |\cO_X(a)|$ be a smooth integral surface, and consider a complex
\[ K: K_1 = U_1 \otimes \cO_D(-m) \xrightarrow{\theta} U_0 \otimes \cO_D(-n) = K_0 \]
as before.
\begin{enumerate}[(a)]
    \item If $\theta^\vee$ is surjective on the support of $T|_D$, then $\H^2(E|^{\L}_D \otimes^{\L} K^\vee) = 0$.
    \item If $D$ meets the support of $T$, then $\H^1(E|^{\L}_D \otimes^{\L} K^\vee) \neq 0$.
    \item If $D$ avoids the support of $T$, but $\Tor_1(F^\dd/F,\cO_D)$ has non-trivial $0$-dimensional torsion on $D$, then $\H^0(E|^{\L}_D \otimes^{\L} K^\vee) \neq 0$.
\end{enumerate}

\end{lem}
\begin{proof}
First we show that $\H^2(\cH^0(E|_D^{\L}) \otimes^{\L} K^\vee) = 0$. As we saw in the proof of \cref{lem:DerivedRestriction}, there is a short exact sequence
\begin{align}\label{eq:T0}
     0 \to F|_D \to \cH^0(E|_D^{\L}) \to \cH^{-1}(T|_D^{\L}) \to 0
\end{align}
with $\cH^{-1}(T|_D^{\L})$ a $0$-dimensional sheaf, and so we obtain a surjective map
\[
    \H^2(F|_D \otimes^{\L} K^\vee) \to  \H^2(\cH^0(E_t|_D^{\L}) \otimes^{\L} K^\vee) \to 0.
\] 
To compute $\H^2(F|_D \otimes^{\L} K^\vee)$ we will use the exact triangle
\[
    F|_D \otimes^{\L} K_0^\vee \to F|_D \otimes^{\L} K^\vee \to F|_D \otimes^{\L} K_1^\vee[-1].
\]
As $K_0$ and $K_1$ are locally free we have $F|_D \otimes^{\L} K_0^\vee = F|_D \otimes K_0^\vee$ and $F|_D \otimes^{\L} K_1^\vee = F|_D \otimes K_1^\vee$. 
Moreover, since $F|_D$ is $n$-regular by our assumptions on $n$, we know that
\begin{align*}
     \rH^i(F|_D \otimes^{\L} K_0^\vee) = \rH^i(F|_D \otimes^{\L} K_1^\vee) = 0 \quad \text{for }i \ge 1.
\end{align*}
Hence the long exact sequence in cohomology induced by the above triangle shows that $\H^2(F|_D \otimes^{\L} K^\vee) = 0$, from which one gets the statement about $\cH^0(E|_D^{\L})$.

Now consider the exact triangle
\[
    \cH^0(E|_D^{\L}) \to E|_D^{\L} \to T|_D[-1]
\]
given by \cref{lem:DerivedRestriction}, which leads to a long exact sequence in cohomology
\[
    \ldots \to \H^1(\cH^0(E|_D^{\L}) \otimes^{\L} K^\vee) \to \H^1(E|_D^{\L} \otimes^{\L} K^\vee) \to \H^0(T|_D \otimes^{\L} K^\vee) \to \ldots.
\]
From this one gets the following maps 
\begin{align}\label{eq:h12}
    \H^2(E|_D^{\L} \otimes^{\L} K^\vee) \cong \H^1(T|_D \otimes^{\L} K^\vee), \quad \H^1(E|_D^{\L} \otimes^{\L} K^\vee) \twoheadrightarrow \H^0(T|_D \otimes^{\L} K^\vee).
\end{align}
Using the spectral sequence
\[
    E_2^{p,q} = \rH^p(\cH^q(T|_D \otimes^{\L} K^\vee)) \Rightarrow \H^{p+q}(T|_D \otimes^{\L} K^\vee)
\]
and the fact that $T|_D$ is a $0$-dimensional sheaf, one obtains
\begin{align*}
    \H^1(T|_D \otimes^{\L} K^\vee) ={}& \rH^0(\cH^1(T|_D \otimes^{\L} K^\vee)) \\
    ={}& \rH^0(\Coker(T|_D \otimes K_0^\vee \xrightarrow{\id \otimes \theta^\vee} T|_D \otimes K_1^\vee)),
\end{align*}
and 
\begin{align*}
    \H^0(T|_D \otimes^{\L} K^\vee) ={}& \rH^0(\cH^0(T|_D \otimes^{\L} K^\vee)) \\
    ={}& \rH^0(\Ker(T|_D \otimes K_0^\vee \xrightarrow{\id \otimes \theta^\vee} T|_D \otimes K_1^\vee)).
\end{align*}
Using \eqref{eq:h12} we easily deduce (a) and (b).

It remains to show (c). Let $Q \coloneqq F^\dd/F$ and consider the short exact sequence
\[ 0 \to F \to F^\dd \to Q \to 0, \]
 which induces an exact triangle
    \[ F|^\LL_D \to F^\dd|^\LL_D \to Q|^\LL_D. \]
Since $F^\dd|^\LL_D \cong F^\dd|_D$ is a torsion-free sheaf, we have $\sH^{-1}(F^\dd|^\LL_D) = 0$, and so the long exact sequence of cohomology sheaves gives an injection
    \[ \sH^{-1}(Q|^\LL_D) \hookrightarrow F|_D = \sH^0(E|^\LL_D). \]
The last equality holds by \eqref{eq:T0}, since $D$ avoids the support of $T$. Now consider the exact triangle
\[
    Q(-a) \to Q \to Q|^\LL_D,
\]
which induces an exact sequence on cohomology sheaves
\[
    0 \to \sH^{-1}(Q|^\LL_D) \to Q(-a) \to Q \to \sH^{0}(Q|^\LL_D) \to 0.
\]
Therefore $\sH^{-1}(Q|^\LL_D) = \Tor_1(Q,\cO_D)$, whose $0$-dimensional torsion $Q'$ is non-trivial by assumption. As above, we obtain
\[
    \H^0(Q'\otimes^{\L} K^\vee) = \rH^0(\Ker(Q' \otimes K_0^\vee \xrightarrow{\id \otimes \theta^\vee} Q' \otimes K_1^\vee)) \neq 0,
\]
and thus
\[
    0 \neq \H^0(Q'\otimes^{\L} K^\vee) \hookrightarrow \H^0(\sH^{-1}(Q|^\LL_D)\otimes^{\L} K^\vee) \hookrightarrow \H^0(\sH^0(E|^\LL_D)\otimes^{\L} K^\vee).
\]
Taking the long exact sequence associated to the triangle
\[
    \cH^0(E|_D^{\L})\otimes^{\L} K^\vee \to E|_D^{\L}\otimes^{\L} K^\vee \to T|_D[-1]\otimes^{\L} K^\vee,
\]
gives an inclusion
\[
    0 \neq \H^0(\cH^0(E|_D^{\L})\otimes^{\L} K^\vee) \hookrightarrow \H^0(E|_D^{\L}\otimes^{\L} K^\vee),
\]
showing (c).
\end{proof}

\begin{comment}
\begin{rmk}\label{effectiveBridgeland}
A careful analysis of the conditions imposed by Álvarez-Cónsul and King on $m, n$ (see \cite[(C:1)–(C:5)]{alvarez2007functorial}) will likely produce \textit{effective} semiampleness results for $\cL_{m,n,a}$, at least in the coprime case, where we already have effective bounds for $a$.
\end{rmk}
\end{comment}

\subsection{A separation result}\label{subsect:SeparationResult}
We next analyze the extent to which the line bundle $\cL_{m,n,a}$ separates points in $\cM_X(v)$. The following notion (introduced in a more general context in \cite[5.10]{grothendieck1965elements}) will be used to state our separation result, \cref{prop:L1fiber}.

\begin{defprop}\label{def:closure}
Let $F \in \Coh(X)$ be a torsion-free sheaf. The \textit{2-closure} $F^{[2]}$ of $F$ is the minimal coherent sheaf $F \subset F' \subset F^\dd$ which has homological dimension at most $1$. In particular, $F$ and $F^{[2]}$ are isomorphic outside finitely many points.
\end{defprop}
\begin{proof}
If $Q\coloneqq F^{\vee \vee}/F$ is zero, then $F^{[2]} = F^\dd$ since $F^\dd$ has homological dimension at most $1$ by the Auslander-Buchsbaum formula. Otherwise, let $Z_F$ be the maximal $0$-dimensional subsheaf of $Q$, and denote by $F'$ the kernel of the composition $F^{\vee \vee} \twoheadrightarrow Q \twoheadrightarrow Q/Z_F$. From the short exact sequence
\[
    0 \to F' \to F^{\vee \vee} \to Q/Z_F \to 0
\]
we obtain by \cite[Tag 00LX]{stacks-project} that
\[
    \depth_x(F') \ge \min\, \{ \depth_x(F^{\vee \vee}), \depth_x (Q/Z_F) + 1\}
\]
for every $x \in X$. As $F^{\vee \vee}$ is reflexive, $\depth_x(F^{\vee \vee}) \ge 2$ whenever $\dim(\cO_{X,x}) \ge 2$. Also, if $\dim(\cO_{X,x}) = 3$, then $\depth_x(Q/Z_F) \ge 1$ since $Q/Z_F$ is pure of dimension 1. Therefore $\depth_x(F') \ge 2$ for every closed point $x \in X$. By using again the Auslander-Buchsbaum formula we deduce that $F'$ has homological dimension at most 1. By construction, $F'$ is the minimal sheaf lying between $F$ and $F^\dd$ with this property; hence $F^{[2]} = F'$ in this case.
\end{proof}

Let $F$ be a $2$-Gieseker-semistable sheaf on $X$. Then $F^{[2]}$ is also $2$-Gieseker-semistable, since it differs from $F$ at only finitely many points. Also, one can easily see that given two Jordan-H\"older filtrations $F_\bullet$ and $F'_\bullet$, there is an isomorphism between $\gr(F_\bullet)^{[2]}$ and $\gr(F'_\bullet)^{[2]}$. In the following, we denote by $\gr(F)^{[2]}$ the 2-closure of the graded module corresponding to any Jordan-H\"older filtration of $F$.

\begin{lem}\label{lem:bound2Closures}
The family of $2$-Gieseker-semistable coherent sheaves
\[
    \{ \, \gr(\cH^0(E))^{[2]} \mid E \in \cM_X(v)(\C) \, \}
\]
is bounded.    
\end{lem}
\begin{proof}
For any $\C$-point $E \in \cM_X(v)$, using the additivity of the Hilbert polynomial in short exact sequences we obtain
\[
    \rP_1(\gr(\cH^0(E))^{[2]}) = \rP_1(\cH^0(E)), \quad a_3(\gr(\cH^0(E))^{[2]}) \ge a_3(\cH^0(E)),
\]
where $\rP_1(-)$ and $a_3(-)$ are defined as in \eqref{eq:hilbPol}. Since the family of sheaves $\cH^0(E)$ with $E \in \cM_X(v)$ is bounded, the family in the statement is also bounded by \cite[Theorem 4.8]{maruyama81on}. 
\begin{comment}
Then there exists an integer $n > 0$ such that all sheaves $F \in \frak{F}$ are $n$-regular, i.e.
\[
    \rH^i(X,F(n-i)) = 0, \quad \text{for }i > 0.
\]
Taking the long exact sequence associated to
\[
    0 \to F \to F^{[2]} \to F^{[2]}/F \to 0,
\]
and using that $F^{[2]}/F$ is $0$-dimensional, we obtain
\[
    \rH^i(X,F^{[2]}(n-i)) = 0, \quad \text{for }i > 0.
\]
This shows that all sheaves $\gr(\cH^0(E))^{[2]}$ with $E \in \cM_X(v)(\C)$ are $n$-regular, from which we get the boundedness result.
\end{comment}
\end{proof}

\begin{cor}\label{cor:finiteRes}
Let $\iota : X \hookrightarrow \P^r$ be the closed embedding given by $\cO_X(1)$. Then there exists an integer $n > 0$ such that for every $\C$-point $E \in \cM_X(v)$, the sheaf $\iota_*\gr(\cH^0(E))^{[2]}$ admits a minimal locally free resolution of the form
\[
    0 \to F_{r-2} \to \ldots \to F_0 \to \iota_*\gr(\cH^0(E))^{[2]} \to 0,
\]
where $F_i = \bigoplus_j \cO_{\P^r}(- a_{i,j})^{\beta_{i,j}}$ with $a_{i,j} \ge 0$,  and
\[
    n \ge \sup \{ \, a_{i,j} - i \mid i \ge 0,\, j \ge 0 \, \}.
\]
\end{cor}
\begin{proof}
By \cref{lem:bound2Closures}, there exists an integer $n > 0$ such that every coherent sheaf $\gr(\cH^0(E))^{[2]}$ with $E \in \cM_X(v)(\C)$ is $n$-regular. The existence of such a minimal resolution as in the statement is then a consequence of \cite[Proposition 4.16]{EisenbudSyzygyies}. Note that in our case we can always choose a resolution of length at most $r-2$. Indeed, by definition, the $2$-closure $\gr(\cH^0(E))^{[2]}$ has homological dimension at most $1$ on $X$, and thus $\iota_* \gr(\cH^0(E))^{[2]}$ has homological dimension at most $r-2$ on $\P^r$.
\end{proof}

\begin{lem}\label{lem:restrictionSequiv} 
There exists an integer $a > 0$ satisfying: for every two $\C$-points $E_1, E_2 \in \cM_X(v)$ such that the $2$-Gieseker-semistable sheaves $\cH^0(E_1), \cH^0(E_2)$ are not $S$-equivalent, their restrictions $E_1|^\LL_D$ and $E_2|^\LL_D$ to a general surface $D \in |\cO_X(a)|$ are Gieseker-semistable sheaves that are not $S$-equivalent on $D$.
\end{lem}
\begin{proof}
Let $a > 0$ be an integer satisfying the condition of \cref{lem:effRestriction}.
We set $$F_1 \coloneqq \cH^0(E_1)^{[2]}, \quad F_2 \coloneqq \cH^0(E_2)^{[2]}.$$ Since $F_{1}$ and $F_{2}$ are 2-Gieseker-semistable, they admit some Jordan-H\"older filtrations $F_{1}^\bullet$ and $F_{2}^\bullet$, respectively, whose factors are $2$-Gieseker-stable. 

Note that $F_1$ and $F_2$ have homological dimension at most 1. By using \cite[Tag 065S]{stacks-project}, which characterizes homological dimension in short exact sequences, we may assume that the filtrations $F_{1}^\bullet, F_{2}^\bullet$ were chosen such that $G_1$ and $G_2$ also have homological dimension at most 1. Hence
\[
    G_{1} \cong \gr(\cH^0(E_1))^{[2]},\quad G_{2} \cong \gr(\cH^0(E_2))^{[2]},
\]
which, by hypothesis, are not isomorphic to each other.

By \cref{lem:effRestriction}, for a general surface $D \in |\cO_X(a)|$, the restrictions $F_{1}|_D \cong E_1|^\LL_D$ and $F_{2}|_D \cong E_2|^\LL_D$ are Gieseker-semistable, and $G_{1}|_D$ and $G_{2}|_D$ are Gieseker-polystable. In particular the restrictions $F_{1}^\bullet|_D$ and $F_{2}^\bullet|_D$ yield Jordan-Hölder filtrations of $F_{1}|_D$ and $F_{2}|_D$, respectively, with respect to the notion of Gieseker-semistability on $D$, implying that
\[
    G_{1}|_D = \gr(F_{1}^\bullet|_D)\quad \text{and}\quad  G_{2}|_D = \gr(F_{2}^\bullet|_D).
\]

Now consider the short exact sequence
\[
    0 \to G_{2}(-D) \to G_{2} \to G_{2}|_D \to 0.
\]
Applying $\Hom_X(G_{1},-)$ to the above sequence, we get the long exact sequence
\begin{align*}
     0 \to{}& \Hom_X(G_{1},G_{2}(-D)) \to \Hom_X(G_{1},G_{2}) \to \Hom_X(G_{1}|_D,G_{2}|_D) \to \\
     \to{}& \Ext^1_X(G_{1},G_{2}(-D)) \to \ldots.
\end{align*}
Since $G_1$ and $G_2$ are $2$-Gieseker-polystable sheaves satisfying $\rp_{1}(G_1) = \rp_{1}(G_2)$, we get
\[
    \rp_{1}(G_1,m) > \rp_{1}(G_2(-D),m) = \rp_{1}(G_2,m-a) \quad \text{for }m \gg 0,
\]
and thus $\Hom_X(G_{1},G_{2}(-D)) = 0$.

Next we show that $\Ext^1_X(G_1,G_2(-D)) = 0$. Let $\iota : X \hookrightarrow \P^r$ be the closed embedding given by $\cO_X(1)$. By \cref{cor:finiteRes}, there exists an integer $n > 0$ such that every member of the bounded family
\[
    \{ \, \gr(\cH^0(E))^{[2]} \mid E \in \cM_X(v)(\C) \, \}
\]
is $n$-regular, and each $\iota_*\gr(\cH^0(E))^{[2]}$ admits a minimal locally free resolution of the form
\[
    0 \to F_{r-2} \to \ldots \to F_0 \to \iota_*\gr(\cH^0(E))^{[2]} \to 0
\]
such that  $F_i = \bigoplus_j \cO_{\P^r}(- a_{i,j})^{\beta_{i,j}}$ with $a_{i,j} \ge 0$ and
\[
    n \ge \sup \{ \, a_{i,j} - i \mid i \ge 0,\, j \ge 0 \, \}.
\]
By increasing $a$ if necessary, we may also assume that 
\begin{equation}\label{eq:aij}
    a_{i,j}  - i + a - 2 \ge n \quad \text{for all } i \ge 0,\, j \ge 0.
\end{equation}
This last condition will be used shortly.

In particular, we get a resolution as above for $G_2 = \gr(\cH^0(E_2))^{[2]}$, which can be split into short exact sequences
\begin{align*}
     &0 \to K_1 \to F_0 \to \iota_* G_2 \to 0, \quad 0\to K_2 \to F_1 \to K_1 \to 0, \quad \ldots \\
     &0 \to F_{r-2} \to F_{r-3} \to K_{r-3} \to 0.
\end{align*}
We set $K_0 \coloneqq \iota_* G_2$ and $K_{r-2} \coloneqq F_{r-2}$. Then every short exact sequence
\[
    0 \to K_{i+1} \to F_i \to K_i \to 0
\]
induces a complex
\begin{equation}\label{eq:exti}
     \Ext_{\P^r}^{i+1}(\iota_*G_1,F_i(-a)) \to \Ext_{\P^r}^{i+1}(\iota_*G_1,K_i(-a))  \to \Ext_{\P^r}^{i+2}(\iota_* G_1,K_{i+1}(-a)) 
\end{equation}
which is exact in the middle.
Using Serre Duality, we get
\begin{align*}
      \Ext^{i+1}_{\P^r}(\iota_* G_1,F_i(-a)) \cong{}& \Ext_{\P^r}^{r - i-1}(F_i(-a),\iota_*G_1(-r-1))^\vee \\ 
    \cong{}& \bigoplus_j \Ext_{\P^r}^{r-i-1}(\cO_{\P^r}, \iota_*G_1(a_{i,j} + a - r - 1))^{\oplus \beta_{i,j}} \\
    \cong{}& \bigoplus_j \rH^{r-i-1}(X, G_1(a_{i,j} + a - r - 1))^{\oplus \beta_{i,j}}.
\end{align*}
From the inequality \eqref{eq:aij} and the $n$-regularity of $G_1$, we obtain that the last group above is trivial for $i \le r-2$. 

Starting from the sequence \eqref{eq:exti} at $i = r-3$ and using that
\[
    \Ext^{i+1}_{\P^r}(\iota_* G_1,F_i(-a)) = 0,
\]
a decreasing induction on $i$ shows that
\[
    \Ext_{\P^r}^{1}(\iota_* G_1,\iota_* G_2(-a)) = 0.
\]
Therefore $\Ext^1_X(G_1,G_2(-D)) = 0$, implying that
\[
    \Hom_X(G_{1},G_{2}) \to \Hom_X(G_{1}|_D,G_{2}|_D)
\]
is an isomorphism. The same holds with $G_{1}$ and $G_{2}$ interchanged, from which we deduce that the restrictions $G_{1}|_D$ and $G_{2}|_D$ are still non-isomorphic. In conclusion,  $E_1|^\LL_D$ and $E_2|^\LL_D$ are Gieseker-semistable sheaves that are not $S$-equivalent on $D$.   
\end{proof}

\begin{lem}\label{lem:avoidAssPoints}
There exists an integer $a > 0$ such that a general divisor $D \in |\cO_X(a)|$ avoids the $1$-dimensional associated points of all sheaves $\cH^0(E)^\dd/\cH^0(E)$ with $E \in \cM_X(v)$.
\end{lem}
\begin{proof}
Since the family of sheaves $\cH^0(E)$ with $E \in \cM_X(v)$ is bounded, then so is the family of quotients $\cH^0(E)^\dd/\cH^0(E)$. Now the result is a consequence of \cite[Lemma 3.8]{pavel2021moduli}.
\end{proof}

In the following, we fix an integer $a > 0$ which satisfies the conclusions of Lemmas \ref{lem:restrictionSequiv} and \ref{lem:avoidAssPoints}, and moreover condition \eqref{cond:a}. Once $a$ is fixed, we choose integers $m, n > 0$ satisfying conditions  \eqref{cond:b} and \eqref{cond:c}, as in the proof of \cref{thm:SemiAmple}.

Then the line bundle $\cL_{m,n,a}$ is semiample over $\cM_X(v)$, cf. \cref{thm:SemiAmple}. Let $S$ be a smooth, proper, connected curve and let $\sE_S \in \Db(S \times X)$ be an $S$-perfect complex corresponding to a map $S \to \cM_X(v)$ such that $\deg(\sL_{m,n,a}|_S)=0$. For each closed point $s \in S$, denote 
\[ E_s = {\sE_S}|^\LL_{\{s\} \times X}, \quad F_s = \sH^0(E_s), \quad T_s = \sH^1(E_s), \quad Q_s = F_s^{\vee \vee}/F_s. \]
We also denote by $\gr(F_s)^{[2]}$ the 2-closure of the graded module corresponding to any Jordan-H\"older filtration (with respect to the notion of 2-Gieseker-semistability) of $F_s$. Recall that $\gr(F_s)^{[2]}$ does not depend on the chosen Jordan-H\"older filtration.

\begin{prop}\label{prop:L1fiber}
Under the above hypotheses, the $2$-Gieseker-semistable sheaves $F_s$ with $s \in S$ lie in the same $S$-equivalence class in $\Coh_{3,1}(X)$, and so $\gr(F_s)^{[2]}$ does not vary with $s \in S$.
\end{prop}
\begin{proof}
Suppose, on the contrary, that there are two points $s_1, s_2 \in S$ such that $F_{s_1}$ and $F_{s_2}$ are not $S$-equivalent. By \cref{lem:restrictionSequiv} and our choice of the integer $a >0$, there exists a general smooth divisor $D \in |\cO_X(a)|$ such that $E_{s_1}|^\LL_D \cong F_{s_1}|_D$, $E_{s_2}|^\LL_D \cong F_{s_2}|_D$ are Gieseker-semistable sheaves that are not $S$-equivalent on $D$. We can choose a general divisor $D \in |\cO_X(a)|$ that avoids the 1-dimensional associated points of $Q_s$ for every $s \in S$, cf. \cref{lem:avoidAssPoints}. Therefore $\Tor_1(\cQ_s,\cO_D)$ is either zero or 0-dimensional on $D$.

As in the proof of \cref{thm:SemiAmple}, there is a complex
\[ 
    K: K_1 = U_1 \otimes \cO_D(-m) \xrightarrow{\theta} U_0 \otimes \cO_D(-n) = K_0
\]
such that $\H^i(E_{s_1}|^\LL_D \otimes^{\LL} K^\vee) = 0$ for all $i$, and $\H^i(E_{s}|^\LL_D \otimes^{\LL} K^\vee) = 0$ for all $i \neq 0,1$ and $s \in S$. Furthermore, if we denote by $\cE_D$ the (derived) restriction of $\cE_S$ to $S \times D$, then there exists a section $\delta_K$ of $\lambda_{\cE_D}([K^\vee])^\vee \cong \lambda_{\cE_D}([K])$ over $S$ non-vanishing at $s_1$. We have also seen in the proof of \cref{thm:SemiAmple} that
\[
    \lambda_{\cE_D}([K])^{\otimes k} \cong \cL_{m,n,a}|_S^{\otimes k'} 
\]
for some integers $k, k' >0$, and so 
\[
    \deg(\lambda_{\cE_D}([K])) = \deg(\cL_{m,n,a}|_S) = 0.
\]
This implies that $\delta_K$ is nowhere vanishing over the curve $S$; thus
$\H^i(\cE_{s}|^\LL_D \otimes^{\LL} K^\vee) = 0$ for all $i$ and $s \in S$. By \cref{lem:sectionOnSurface} we obtain that for every $s\in S$, the divisor $D$ avoids the support of $T_s$ and $\Tor_1(Q_s,\cO_D) = 0$. The last condition implies that $F_s|_D$ injects into the pure sheaf $F_s^{\dd}|_D$. Indeed, this follows from the exact sequence
\[
    0 \to \Tor_1(Q_s,\cO_D) \to F_s|_D \to F^\dd_s|_D \to Q_s|_D \to 0.
\]
In particular, each restriction $F_s|_D$ remains torsion-free over $D$. Moreover, by our choice of the integer $n >0$, the restrictions $F_s|_D$ with $s \in S$ are \textit{$n$-regular} pure sheaves of fixed Hilbert polynomial $P$ on $D$. Therefore $\cE_{s}|^\LL_D \cong F_s|_D$ is an $n$-regular pure sheaf on $D$ for all $s \in S$. By applying \cref{thm:Gss} we obtain that $\cE_{s}|^\LL_D$ is also Gieseker-semistable, showing that $\cE_D$ is quasi-isomorphic to an $S$-flat family of Gieseker-semistable sheaves on $D$. 

Now consider the moduli stack $\cM^\mathrm{Gss}_D(P)$ of Gieseker-semistable sheaves of Hilbert polynomial $P$ on $D$. This stack admits a good moduli space $$\pi : \cM^\mathrm{Gss}_D(P) \to \rM^\mathrm{Gss}_D(P),$$
see \cite[Example 7.28]{alper2019existence}. If we denote by $\cF$ the universal family of sheaves over $\cM^\mathrm{Gss}_D(P) \times D$, then the line bundle $\lambda_\cF([K])$ descends to an \textit{ample} line bundle $\lambda([K])$ on $\rM^\mathrm{Gss}_D(P)$. Indeed, $\lambda([K])$ coincides with line bundle $\lambda_U(P)$ associated to the pair $(U_0,U_1)$, constructed in \cite[Proposition 7.7]{alvarez2007functorial}.

Consider the composition
\[
    S \xrightarrow{f} \cM^\mathrm{Gss}_D(P) \xrightarrow{\pi} \rM^\mathrm{Gss}_D(P),
\]
where the first map corresponds to the $S$-flat family $\cE_D$. Letting $g = \pi \circ f : S \to \rM^\mathrm{Gss}_D(P)$, we obtain
\[
    g^*\lambda([K]) \cong f^*\lambda_\cF([K]) \cong \lambda_{\cE_D}([K]).
\]
Finally, since $E_{s_1}|^\LL_D$, $E_{s_2}|^\LL_D$ are not $S$-equivalent as Gieseker-semistable sheaves on $D$, the images of $s_1$ and $s_2$ through $g$ are distinct points in $\rM^\mathrm{Gss}_D(P)$. Indeed, this follows since the geometric points of $\rM^\mathrm{Gss}_D(P)$ correspond to $S$-equivalence classes of Gieseker-semistable sheaves on $D$, cf. \cite[Theorem 4.3.4]{HL}. Therefore there exists a global section of some power of the ample line bundle $\lambda([K])$ separating $g(s_1)$ and $g(s_2)$. The pullback of this section via $g$ produces a global section of some power of $\lambda_{\cE_D}([K])$ separating $s_1$ and $s_2$ over $S$. This yields a contradiction since we have seen above that $\deg(\lambda_{\cE_D}([K])) = 0$.
\end{proof}

\section{PT-semistable objects}\label{sect:PT-semistable objects}

%%%%%%%%%%%%%%%%%%%%%%%%%%%%%%%%%%%%%%%%%%%%%%%%%%%%%
%%%%%%%%%%%%%%%%%%%%%%%%%%%%%%%%%%%%%%%%%%%%%%%%%%%%%
%%%%%%%%%%%%%%%%%%%%%%%%%%%%%%%%%%%%%%%%%%%%%%%%%%%%%
In this section, we recall the definition and basic properties of PT-stability conditions. They are examples of polynomial stability conditions defined in \cite{bayer-polynomial} as a generalization of Bridgeland stability conditions in order to understand the large volume limit of Bridgeland stability, as well as to study the relation between the DT/PT invariants. We largely follow \cite{lo-PT1} and \cite{lo-PT2} in our presentation, except that we use a slightly different convention for the category of perverse sheaves (see \cref{rmk:convention0}).

As before, $X$ is a smooth, projective, connected threefold over $\C$, and fix a very ample divisor $H \in |\cO_X(1)|$. In the following, we let $\phi(z) \in (0, \pi]$ denote the phase of a complex number $z$ in the extended upper half-plane
\[ \Hh = \{\, z \in \C \mid \im z > 0 \,\} \cup \R_{<0}. \]

\begin{defn}\label{defn:PTstab}
    A \textbf{PT-stability condition} on $X$ consists of the data of the heart $\sA^p(X)$ together with a group homomorphism $Z: \Kn(X) \to \C[m]$, called the \emph{central charge}, of the form
    \[ Z(E)(m) = \sum_{d=0}^3 \rho_d \left(\int_X H^d \cdot \ch(E) \cdot U\right) m^d, \]
    where
    \begin{enumerate}[(a)]
        \item the $\rho_d \in \C^*$ are nonzero complex numbers such that $-\rho_0, -\rho_1, \rho_2, \rho_3 \in \Hh$, and whose phases satisfy
        \[ \phi(\rho_2) > \phi(-\rho_0) > \phi(\rho_3) > \phi(-\rho_1). \]
        \item $U = 1 + U_1 + U_2 + U_3 \in A^*(X)$ is a class with $U_i \in A^i(X)$ for $i = 1, 2, 3$.
    \end{enumerate}
\end{defn}
The configuration of the complex numbers $\rho_i$ is compatible with the heart $\sA^p(X)$ in the sense that for any nonzero $E \in \sA^p(X)$, we have $Z(E)(m) \in \Hh$ for $m \gg 0$. This allows us to define a notion of stability on $\sA^p(X)$: an object $E \in \sA^p(X)$ is called $Z$-\textbf{stable} (resp. $Z$-\textbf{semistable}) if for every proper nonzero subobject $F \subs E$, we have 
\[ \phi(Z(F)(m)) < \phi(Z(E)(m) \quad (\mathrm{resp.} \quad \phi(Z(F)(m)) \le \phi(Z(E)(m)) \quad \mathrm{for} \quad m \gg 0.  \]

\begin{rmk}\label{rmk:convention}
As we noted in \cref{rmk:convention0}, our definition of the heart $\sA^p(X)$ differs from that in the literature by a shift. To account for this, our definition of the charge $Z$ also differs in that $-\rho_0, -\rho_1, \rho_2, \rho_3$, rather than $\rho_0, \rho_1, -\rho_2, -\rho_3$, are in $\Hh$.  This will let us view the moduli of PT-semistable objects as an enlargement of the moduli of $\mu$-stable reflexive sheaves without having to perform a shift.
\end{rmk} 

In \cite{PT}, Pandharipande and Thomas define a \emph{stable pair} on $X$ to be a map of the form
\[ \Oh_X \xrightarrow{s} F, \]
where $F$ is a sheaf of pure dimension 1 and the cokernel of $s$ is $0$-dimensional. In \cite[Proposition 6.1.1]{bayer-polynomial}, Bayer shows that for any PT-stability condition, the stable objects in $\sA^p(X)$ with numerical invariants $\ch = (1, 0, -\be, -n)$ and trivial determinant coincide precisely with these stable pairs. The following partial characterization of PT-semistable objects generalizes this fact to higher rank.

\begin{prop}[{\cite[Lemma 3.3]{lo-PT1}, \cite[Proposition 2.24]{lo-PT2}}]\label{prop:PT-partial-characterization}
    If $v \in \Kn(X)$ is a class of rank $\rk(v) > 0$, then any PT-semistable object $E \in \sA^p(X)$ of class $v$ satisfies the following conditions:
    \begin{enumerate}[(a)]
        \item $\sH^0(E)$ is $2$-Gieseker-semistable,
        \item $\sH^1(E)$ is $0$-dimensional,
        \item $\Hom_{\Db(X)}(T[-1], E) = 0$ for any $0$-dimensional sheaf $T$.
    \end{enumerate}
    In particular $\sH^0(E)$ is also $\mu$-semistable. If moreover $\rk(v)$ and $H^2 \cdot \ch_1(v)$ are coprime, then any object $E$ of class $v$ in $\sA^p(X)$ satisfying these conditions is PT-stable, $\sH^0(E)$ is $\mu$-stable, and there are no strictly semistable objects.
\end{prop}

We note that for any PT-semistable object $E$ of positive rank, the torsion-free sheaf $\sH^0(E)$ has homological dimension at most $1$, cf. \cite[Example 5.5]{Lo2021PTstableQuotients}.

We will also need the following observation.
\begin{lem}\label{subobjposrank}
    Let $E \in \sA^p(X)$ be a PT-semistable object with respect to a charge $Z: \Kn(X) \to \C[m]$ as in Definition \ref{defn:PTstab}, and assume $\rk(E) > 0$. If $F \subs E$ is a subobject in $\sA^p(X)$ such that 
    \[ \phi(Z(F)(m)) = \phi(Z(E)(m)) \quad \text{for } m \gg 0, \]
    then $\rk(F) > 0$.
\end{lem}
\begin{proof}
    Let $Q$ denote the cokernel of the inclusion $F \subs E$ in $\sA^p(X)$, so that we have a short exact sequence
    \[ 0 \to F \to E \to Q \to 0 \]
    in $\sA^p(X)$. This induces an exact sequence
    \[ 0 \to \sH^0(F) \to \sH^0(E) \to \sH^0(Q) \to \sH^1(F) \to \sH^1(E) \to \sH^1(Q) \to 0 \]
    in $\Coh(X)$. Since by \cref{prop:PT-partial-characterization}, the sheaf $\sH^0(E)$ is torsion-free, if $\rk(F) = 0$, then $F = F'[-1]$, where $F' = \sH^1(F)$ is a coherent sheaf with $\dim(\Supp(F')) \le 1$. 
    
    If $\dim(\Supp(F')) = 1$, then
    \[ \lim_{m \to \infty} \phi(Z(F)(m)) = \phi(\rho_1) < \phi(-\rho_3) = \lim_{m \to \infty} \phi(Z(E)(m)). \]
    Similarly, if $\dim(\Supp(F')) = 0$, then
    \[ \lim_{m \to \infty} \phi(Z(F)(m)) = \phi(\rho_0) > \phi(-\rho_3) = \lim_{m \to \infty} \phi(Z(E)(m)). \]
    In neither case can we have $\phi(Z(F)(m)) = \phi(Z(E)(m))$ for $m \gg 0$. 
\end{proof}

\subsection{Moduli of PT-semistable objects}\label{setup}

The theory of moduli of PT-semistable objects was developed by Lo in \cite{lo-PT1} and \cite{lo-PT2}, culminating in \cite[Theorem 1.1]{lo-PT2}, where the author constructs the moduli stack of PT-semistable objects of fixed Chern character as a universally closed algebraic stack of finite type, and, in the absence of strictly semistable objects, as a proper algebraic space. 

Here we revisit these results through the framework of good moduli spaces, developed by Alper in \cite{AlperGMS}. To recall the definition, let $\sM$ be an algebraic stack and $\pi: \sM \to M$ a quasi-compact, quasi-separated morphism  to an algebraic space $M$. Then $\pi: \sM \to M$ is called a {\bf good moduli space} if the pushforward functor $\pi_*: \Qcoh(\sM) \to \Qcoh(M)$ is exact, and the natural map $\Oh_M \to \pi_*\Oh_\sM$ is an isomorphism.
\begin{comment} We list a few basic properties of good moduli spaces.
\begin{prop}
    If $\pi: \sM \to M$ is a good moduli space, then the following hold. \begin{itemize}
        \item $\pi$ is surjective and universally closed.
        \item $\pi$ induces a bijection of closed points.
        \item $\pi$ is universal for maps to algebraic spaces.
        \item For every geometric point $x: \Spec \overline{k} \to \sM$ with closed image, the stabilizer group $G_x$ is linearly reductive.
        \item If $\sM$ is locally Noetherian, then so is $M$, and $\pi_*$ preserves coherence.
        \item If $\sM$ is of finite type over a field, then so is $M$.
    \end{itemize}
\end{prop}
\end{comment}

%\subsubsection{Moduli of PT-semistable objects} 
We now establish our moduli setup for PT-semistable objects. Let $v \in \Kn(X)$ be a class of positive rank, and let $Z: \Kn(X) \to \C[m]$ define a PT-stability condition on $X$. The moduli stack of PT-semistable objects of class $v$ is defined to be the substack $\sM^{\mathrm{PT}}_Z(v) \subset \cD_{\cA^p}$ that to a scheme $S$ of finite type over $\C$ associates the groupoid of objects $E \in \Db(S \times X)$ such that 
\begin{enumerate}
    \item $E$ is relatively perfect over $S$, see \cite[Tag 0DHZ]{stacks-project}, and
    \item for all $\C$-points $s \in S$, the derived restriction $E|^\LL_{\{s\}\times X}$ to the fiber over $s$ lies in $\sA^p(X)$, is semistable with respect to $Z$, and has numerical class $v \in \Kn(X)$. 
\end{enumerate}

By \cite[Theorem 1.1]{lo-PT2}, the stack $\sM^{\mathrm{PT}}_Z(v)$ is universally closed and of finite type over $\C$, and moreover, in the case of coprime $\rk(v)$ and $H^2 \cdot \ch_1(v)$,  it is a $\G_m$-gerbe over a proper algebraic space $\rM^{\mathrm{PT}}_Z(v)$. In particular, $\rM^{\mathrm{PT}}_Z(v)$ is both a coarse and a good moduli space for $\sM^{\mathrm{PT}}_Z(v)$ in this case.

However, in the presence of strictly semistable objects we do not know if $\sM^{\mathrm{PT}}_Z(v)$ admits a good moduli space -- since the heart $\sA^p(X)$ is not noetherian, the machinery of \cite[Chapter 7]{AHLH} does not readily apply.

\subsection{Projectivity of moduli of PT-stable objects}\label{sect:projPT}
In this section, we show the projectivity of the good moduli space $\rM^{\mathrm{PT}}_Z(v)$ in the case of coprime rank and degree. We begin with some general considerations that do not require the coprimality assumption.

Let $\sE$ be the universal complex on $\sM^{\mathrm{PT}}_Z(v) \times X$, so that we have the diagram
\begin{center}
    \begin{tikzpicture}
    \matrix (m) [matrix of math nodes, row sep=1em, column sep=1em]
    { & \sE &  \\
    & \sM^{\mathrm{PT}}_Z(v) \times X & \\
    \sM^{\mathrm{PT}}_Z(v) & & X \\};
    \path[dotted]
    (m-1-2) edge node[auto,swap] {$ $} (m-2-2)
    ;
    \path[->] 
    (m-2-2) edge node[auto,swap] {$ p $} (m-3-1)
    (m-2-2) edge node[auto] {$ q $} (m-3-3)
    ;
    \end{tikzpicture}
\end{center}
We define the following line bundles on $\sM^{\mathrm{PT}}_Z(v)$ using the Donaldson morphism (see \cref{section:determinantal-bundles-on-PT}). Let
\[
    \cL_{m,n,a} \coloneqq \lambda_\cE(w_{m,n,a} \cdot [\cO_D]) \in \Pic(\sM^{\mathrm{PT}}_Z(v)),
\]
where $w_{m,n,a}$ is defined by \eqref{eq:grothClass} for some integers $m,n,a > 0$, and $D$ is any divisor in $|\cO_X(a)|$. As in \cite[Example 8.1.8 (iii)]{HL}, we set
\begin{align*}
%v_1(v) = {}&- \chi(v \cdot h^3)h + \chi(v \cdot h)h^3 \in \rK(X),\\
v_0(v) ={}& - \chi(v \cdot h^3)[\cO_X] + \chi(v)h^3 \in \rK(X),
\end{align*}
where $h \coloneqq [\Oh_H] \in \rK(X)$, and define
\[
\cL_0 = \lambda_{\cE}(v_0(v)) \in \Pic(\cM^{\mathrm{PT}}_Z(v)).
\]

\begin{prop}\label{prop:semiAmplePT}
There are integers $m,n,a > 0$ such that 
\begin{enumerate}[(a)]
    \item the line bundle $\cL_{m,n,a}$ is semiample over $\cM^{\mathrm{PT}}_Z(v)$.
    \item for every morphism $S \to \sM^{\mathrm{PT}}_Z(v)$ from a smooth, proper, connected curve $S$ with $\deg(\cL_{m,n,a}|_S = 0)$, if $E_1, E_2 \in \sM^{\mathrm{PT}}_Z(v)$ are $\C$-points lying in the image of $S$, then $$\gr(\cH^0(E_1))^{[2]} \cong \gr(\cH^0(E_2))^{[2]}.$$
\end{enumerate}
\end{prop}
\begin{proof}
By the characterization of PT-semistable objects in \cref{prop:PT-partial-characterization} and the openness of PT-semistability \cite[Proposition 3.3]{lo-PT2}, we see that $\sM^{\mathrm{PT}}_Z(v)$ is an open substack of the stack $\cM_X(v)$ introduced in \cref{sect:moduli spaces}. Therefore the result is a consequence of \cref{thm:SemiAmple} and \cref{prop:L1fiber}.
\end{proof}

Next we show that the above line bundles descend to the good moduli space $\rM^{\mathrm{PT}}_Z(v)$, assuming that this space exists and that $U$ in the definition of $Z$ is the Todd class of $X$. Recall the following criterion \cite[Theorem 10.3]{AlperGMS} for a locally free sheaf $\sF$ on $\sM$ to descend to the good moduli space $M$. 
\begin{prop}\label{vbtogms}
    If $\pi: \sM \to M$ is a good moduli space and $\sM$ is locally Noetherian, then the pullback morphism $\pi^*: \Coh(M) \to \Coh(\sM)$ induces an equivalence of categories between locally free sheaves on $M$ and those locally free sheaves $\sF$ on $\sM$ such that for every geometric point $x: \Spec k \to \sM$ with closed image, the induced representation $x^*\sF$ of the stabilizer $G_x$ is trivial.
\end{prop}

We will also need the following auxiliary result.

\begin{lem}\label{subobjintlemma}
    If $E$ is a PT-semistable of class $v$ and $F \subs E$ is a subobject in $\sA^p(X)$ such that 
    \[ \phi(Z(F)(m)) = \phi(Z(E)(m)) \quad \text{for } m \gg 0, \]
    then
    \begin{equation}\label{subobjintegral}
         \int_X H^d \cdot \ch(F) \cdot U = \frac{\rk(F)}{\rk(E)} \int_X H^d\cdot \ch(E) \cdot U
    \end{equation}
    for $d = 0, 1, 2$.
\end{lem}
\begin{proof}
The assumption on $F$ is equivalent to saying that $p_v(F) = 0$, where $p_v$ is defined by
\begin{equation*}%\label{realvaluedcharge}
    p_v: \rK(X) \to \R[m], \quad p_v(F) = \im Z(v) \re Z(F) - \re Z(v) \im Z(F).
\end{equation*}
To lighten the notation, for every $G \in \rK(X)$ we set  \[ I_d(G) = \int_X H^d \cdot \ch(G) \cdot U, \quad d = 0, \ldots, 3. \]
    We note that since
    \[ I_3(G) = \int_X H^3 \cdot \ch(G) \cdot U = \deg(X) \rk(G), \]
    and since $\rk(F) > 0$ by \cref{subobjposrank}, equation \eqref{subobjintegral} is equivalent to 
    \[ I_3(E) I_d(F) = I_3(F) I_d(E). \]
    Moreover, we set
    \[ r_{ij} = \re(\rho_i) \im(\rho_j), \quad i, j = 0, \ldots, 3, \]
    and note that since none of the complex numbers $\rho_i$ are collinear, the real numbers $r_{ij} - r_{ji}$ are all nonzero for $i \neq j$.
    
    The condition $p_v(F) = 0$ can now be written as
    \begin{equation}\label{rIsum}
        \sum_{d=0}^6 \sum_{i+j = d} r_{ij} ( I_i(E) I_j(F) - I_i(F) I_j(E)) m^d = 0.
    \end{equation}
    We compare coefficients on both sides of this equation. First, the $m^5$ term in \eqref{rIsum} gives
    \[ (r_{32}-r_{23})(I_2(E) I_3(F) - I_2(F) I_3(E)) = 0, \]
    so dividing by $\deg(X)$ and $r_{32}-r_{23}$ gives \eqref{subobjintegral} for $d = 2$. Similarly from the $m^4$ term, noting that the $i = j = 2$ term cancels out, we get
    \[ (r_{31}-r_{13})(I_1(E) I_3(F) - I_1(F) I_3(E)) = 0, \]
    giving \eqref{subobjintegral} for $d = 1$. Finally, the $m^3$ term gives
    \[ (r_{30}-r_{03})(I_0(E) I_3(F) - I_0(F) I_3(E)) + (r_{21}-r_{12})(I_1(E) I_2(F) - I_1(F) I_2(E) = 0. \]
    The second term on the left cancels since $I_3(E) \neq 0$ and by what we have already shown, we have
    \[ I_3(E) I_1(E) I_2(F) = I_3(F) I_1(E) I_2(E) = I_3(E) I_1(E) I_2(E). \]
    Thus, we obtain \eqref{subobjintegral} for $d = 0$, completing the proof.
\end{proof}

\begin{prop}\label{L2descendstogms}
Let $Z$ be a PT-stability function on $\sA^p(X)$ and assume that $U = \td_X$ in the definition of $Z$ is the Todd class of $X$. Assuming a good moduli space $\rM^{\mathrm{PT}}_Z(v)$ exists, the line bundles $\cL_0$ and $\cL_{m,n,a}$ descend to $\rM^{\mathrm{PT}}_Z(v)$.
\end{prop}
\begin{proof}
We show, more generally, that any line bundle of the form $$\cL \coloneqq \lambda_\cE(u) \in \Pic(\sM^{\mathrm{PT}}_Z(v))$$
descends to the good moduli space, whenever $u \in \rK(X)$ is a class satisfying
\begin{enumerate}
    \item $\chi(v \cdot u) = 0$,
    \item $u$ is a linear combination of powers of $h$ in $\rK(X)$.
\end{enumerate}
Let $x \in \rM^{\mathrm{PT}}_Z(v)$ be a closed point corresponding to the $Z$-polystable object
\[ E = \bigoplus_j F_j, \quad \text{where } p_v(F_j) = 0 \text{ for all } j. \]
By \cite[Proposition 4.2]{tajakka2022uhlenbeck}, the automorphism group of $E$ acts trivially on the fiber of $\sL$ at $x$ if and only if $\chi([F_j]\cdot u) = 0$ for each $j$. 
Since $u$ is a linear combination of powers of $h$ and
\[ \ch(h) = \ch(\Oh_X) - \ch(\Oh_X(-H)) = H - \frac{1}{2} H^2 + \frac{1}{6} H^3,\]
we obtain by \cref{subobjintlemma} and linearity that
\[ \int_X \ch(F_j) \ch(u) \td_X = \frac{\rk(F_j)}{\rk(E)} \int_X \ch(E) \ch(u) \td_X. \]
Since $[E] = v \in \Kn(X)$, by the Hirzebruch-Riemann-Roch formula again,
\[ \int_X \ch(E) \ch(u) \td_X = \chi(v \cdot u) = 0. \] 

We now check that the classes $v_0(v)$ and $w_{m,n,a} \cdot [\cO_D]$ satisfy conditions (1) and (2) above. Clearly $v_0(v)$ satisfies (2), and they both satisfy (1): 
\[
    \chi(v \cdot v_0(v)) = 0, \quad \chi(v \cdot w_{m,n,a} \cdot [\cO_D]) = 0.
\]
Using
\[
[\cO_X(m)] = \sum_{i \ge 0} {m + i -1 \choose i}  h^i, \quad [\cO_X(n)] = \sum_{i \ge 0} {n + i -1 \choose i}  h^i
\]
and 
\[
    [\cO_D] = ah - {a \choose 2}h^2 + {a \choose 3}h^3,
\]
we obtain by linearity that $w_{m,n,a} \cdot [\cO_D]$ also satisfies (2).
\end{proof}

In the following, we assume that $\rk(v)$ and $H^2 \cdot \ch_1(v)$ are coprime. By the characterization of PT-stable objects from \cref{prop:PT-partial-characterization}, the moduli stack $\sM^\mathrm{PT}_Z(v)$ does not depend on the class $U \in A^*(X)$ in the definition of the PT-stability function $Z$. Therefore we may assume that $U = \td_X$ to ensure that the line bundles $\cL_0$ and $\cL_{m,n,a}$ descend to line bundles $L_0$ and $L_{m,n,a}$ on the good moduli space $\rM^\mathrm{PT}_Z(v)$, cf. \cref{L2descendstogms}. 

Let $m,n,a > 0$ be integers as in \cref{prop:semiAmplePT}. In particular, the line bundle $\cL_{m,n,a}$ is semiample over $\sM^\mathrm{PT}_Z(v)$, and consequently, $L_{m,n,a}$ is semiample on $\rM^\mathrm{PT}_Z(v)$. Then, there exists $k > 0$ such that $L_{m,n,a}^{\otimes k}$ is globally generated, inducing a map
\[
\rM^\mathrm{PT}_Z(v) \xrightarrow{\eta} \P \coloneqq \P\Gamma(\rM^\mathrm{PT}_Z(v),L_{m,n,a}^{\otimes k}).
\]
By taking $k > 0$ sufficiently large, we may also assume that $\eta$ has connected fibers (see \cite[Lemma 2.1.28]{lazarsfeld2017positivity}). Note that the vector space $\Gamma(\rM^{\mathrm{PT}}_Z(v), L_{m,n,a}^{\otimes k})$ is finite-dimensional since $\rM^{\mathrm{PT}}_Z(v)$ is proper.

We fix a $\C$-point $p \in \P$, and consider the following commutative diagram
\[
    \begin{tikzcd}
        \sM^\mathrm{PT}_Z(v) \ar[r] & \rM^\mathrm{PT}_Z(v) \ar[r,"\eta"] & \P \\
        \sN \ar[r,"\pi"] \ar[u,hookrightarrow] & N \ar[r] \ar[u,hookrightarrow] & \Spec \C \ar[u,"p"]
    \end{tikzcd}
\]
where $\sN$, resp. $N$, denotes the fiber of $\sM^\mathrm{PT}_Z(v)$, resp. $\rM^\mathrm{PT}_Z(v)$, over $p$. In particular the canonical map $\pi: \sN \to N$ is a proper good moduli space.

Moreover, by \cref{prop:PT-partial-characterization}, for every $\C$-point $[E] \in \sM^\mathrm{PT}_Z(v)$ we know that $\cH^0(E)$ is $\mu$-stable of homological dimension at most 1. As we are in the coprime case, applying \cref{prop:semiAmplePT} (b) shows that there exists a $\mu$-stable sheaf $F$ of homological dimension 1 such that every $\C$-point $s \in \sN$ is represented by an exact triangle of the form
\[
F \to E_s \to T_s[-1]
\]
where $T_s$ is 0-dimensional. Indeed, since $N$ is a proper and connected algebraic space, it can be covered by smooth, proper and connected curves $S \to N$, allowing us to apply \cref{prop:semiAmplePT} (b).

Set $r \coloneqq \rk(v)$ and $\beta \coloneqq \ch_1(v)$. In the following we denote by $\sM^{\rm{PT}}_Z(r,\beta)$, resp. $\rM^{\rm{PT}}_Z(r,\beta)$, the disjoint union of all moduli stacks $\sM^{\rm{PT}}_Z(c)$, resp. all corresponding good moduli spaces $\rM^{\rm{PT}}_Z(c)$, over all numerical classes $c \in \Kn(X)$ with $\rk(c) = r$ and $\ch_1(c)=\beta$.

\begin{prop}\label{prop:L0Ampleness}
If $\rk(v)$ and $H^2 \cdot \ch_1(v)$ are coprime, then $N$ is a scheme and the line bundle $L_0$ is ample over $N$.
\end{prop}
\begin{proof}
Denote by $\Quot(\sExt^1(F,\cO_X))$ the Quot scheme of zero-dimensional quotients of $\sExt^1(F,\cO_X)$ on $X$. By \cite[Proposition 5.5]{Beentjes2021DT/PT} there exists a closed embedding $\psi_F$ described by the following diagram
\[
    \xymatrix@C=15pt@R=15pt{ \Quot(\sExt^1(F,\cO_X)) \ar[rr]^{\psi_F} \ar[rd]_{\varphi_F} &  &  \rM^{\mathrm{PT}}_Z(r,\beta) \\  & \cM^{\mathrm{PT}}_Z(r,\beta) \ar[ru]  &  }
\]
sending a geometric quotient $[q: \sExt^1(F,\cO_X) \to Q]$ in $\Quot(\sExt^1(F,\cO_X))$ to a suitable PT-stable object $$F \to E \to Q^D[-1]$$ where $Q^D = \sExt^3(Q,\cO_X)$ is the dual of $Q$ on $X$ (see below Remark 5.2 in \cite{Beentjes2021DT/PT} for the construction). We note that the characterization of PT-stable objects from \cref{prop:PT-partial-characterization} is necessary for the map $\psi_F$ to have codomain $\rM^{\mathrm{PT}}_Z(r,\beta)$. 
Then the base change of $\psi_F$ via the closed subspace $N \subset \rM^{\mathrm{PT}}_Z(r,\beta)$ gives a commutative diagram
\[
    \xymatrix{ Q_1 \ar[rr]^{\psi_{F,N}} \ar[rd]_{\varphi_{F,N}} &  &  N \\  & \sN \ar[ru]_{\pi}  &  }
\]
 where $Q_1 \subset \Quot(\sExt^1(F,\cO_X))$ is the induced closed subscheme and $\psi_{F,N}$ is a closed embedding. Moreover, note that $\psi_{F,N}$ is surjective since any $\C$-point $s \in N$ corresponding to a PT-stable object
\[
    F \to E_s \to T_s[-1]
\]
gives rise to a surjection $\sExt^1(F,\cO_X) \to Q$, with $Q = \sExt^3(T_s,\cO_X)$. Indeed, this can be seen by applying $\sHom(-,\cO_X)$ to the above exact triangle and then looking at the induced long exact sequence.

Let $q : Q_1 \times X \to X$ denote the natural projection, and let $\cQ$ denote the universal quotient on $Q_1 \times X$. Then the pullback $\cG$ of the universal complex $\cE$ on $\sN \times X$ via $\varphi_{F,N}$ fits in an extension
\[
q^*F \to \cG \to \cQ^D[-1]
\]
over $Q_1 \times X$, viewed as an element in $\Ext^1_{Q_1 \times X}(\cQ^D[-1],q^*F)$, cf. proof of \cite[Proposition 5.5]{Beentjes2021DT/PT}. We obtain 
\begin{align*}
    \psi_{F,N}^*(L_0|_N)\cong{}&  \varphi_{F,N}^*(\cL_0|_{\sN})\\
    \cong{}& \lambda_{\varphi_{F,N}^*\cE}(v_0(v)) \quad{}& (\text{by \cite[Lemma 8.1.2 ii)]{HL}})\\
    \cong{}&  \lambda_{q^*F}(v_0(v)) \otimes \lambda_{\cQ^D}(v_0(v))^\vee \quad{}& (\text{by \cite[Lemma 8.1.2 i)]{HL}})\\
    \cong{}&  \lambda_{\cQ}(v_0(v))^\vee,
\end{align*}
where the last isomorphism follows by using the fact that $\lambda_{q^*F}(v_0(v))$ is trivial and the duality
\[
    \Ext^1_{Q_1 \times X}(\cQ^D[-1],q^*F) \cong \Ext^1_{Q_1 \times X}((q^*F)^\vee[1],\cQ[-1]).
\]
Similarly
\[
    \psi_{F,N}^*(L_{m,n,a}|_N) \cong \lambda_{\cQ}(w_{m,n,a} \cdot [\cO_D])^\vee.
\]
Since $\cQ$ is a flat family of $0$-dimensional sheaves, we also have that $\lambda_\cQ(h^k)$ is trivial for $k > 1$. We obtain that
\[
    \psi_{F,N}^*(L_0|_N) \cong \lambda_{\cQ}([\cO_X])^{\chi(v \cdot h^3)},
\]
and
\[
    \psi_{F,N}^*(L_{m,n,a}|_N) \cong \lambda_\cQ(h)^b
\]
for some integer $b$ (as $w_{m,n,a} \cdot [\cO_D]$ is a linear combination of $h, h^2, h^3$ in $\rK(X)$). However, note that $L_{m,n,a}^{\otimes k}$ is trivial over $N$, and thus $\lambda_\cQ(h)^{kb}$ is also trivial.  
Since $$\lambda_\cQ([\cO_X(\ell)]) \cong \lambda_\cQ([\cO_X]) \otimes \lambda_\cQ(h)^{\ell}$$ is ample on $Q_1$ for $\ell > 0$ sufficiently large, cf. \cite[Examples 8.1.3 ii)]{HL}, we deduce that $\lambda_\cQ([\cO_X])$ is also ample over $Q_1$. 

Finally, since $\psi_{F,N}$ is a finite surjective morphism and
\[
    \psi_{F,N}^*(L_0) \cong \lambda_\cQ([\cO_X])^{\chi(v \cdot h^3)}, 
\]
is ample over $Q_1$, it follows by \cite[0GFB]{stacks-project} that $N$ is a scheme and $L_0$ is ample on $N$.
\end{proof}

\begin{thm}\label{thm:MainAmpleness}
Assume that $\rk(v)$ and $H^2 \cdot \ch_1(v)$ are coprime, and let $m,n,a > 0$ be integers as in \cref{prop:semiAmplePT}. Then the line bundle $L_0 \otimes L_{m,n,a}^{\otimes b}$ is ample on $\rM^{\mathrm{PT}}_Z(v)$ for some $b > 0$. In particular, $\rM^{\mathrm{PT}}_Z(v)$ is a projective scheme.
\end{thm}
\begin{proof}
As above, we consider the proper map $\eta : \rM^{\mathrm{PT}}_Z(v) \to \P$ induced by the globally generated line bundle $L_{m,n,a}^{\otimes k}$ for some sufficiently large integer $k > 0$. By \cref{prop:L0Ampleness} the line bundle $L_0$ is ample on each fiber of $\eta$, which implies that $L_0$ is $\eta$-ample, cf. \cite[Tag 0D3A]{stacks-project}. Moreover, we obtain by \cite[Tag 0D32]{stacks-project} that $\rM^{\mathrm{PT}}_Z(v)$ is a scheme. 

Now, by \cite[Proposition 1.7.10]{lazarsfeld2017positivity}, the line bundle
\[
    L_0 \otimes \eta^*\cO_{\P}(1)^{\otimes b}
\]
is ample on $\rM^{\mathrm{PT}}_Z(v)$ for $b > 0$ sufficiently large. By construction $\eta^*\cO_{\P}(1)$ is isomorphic to $L_{m,n,a}^{\otimes k}$, from which we deduce that $L_0 \otimes L_{m,n,a}^{\otimes kb}$ is ample on $\rM^{\mathrm{PT}}_Z(v)$. 
\end{proof}

%%%%%%%%%%%%%%%%%%%%%%%%%%%%%%%%%%%%%%%%%%%%%%%%%%%%%
%%%%%%%%%%%%%%%%%%%%%%%%%%%%%%%%%%%%%%%%%%%%%%%%%%%%%
%%%%%%%%%%%%%%%%%%%%%%%%%%%%%%%%%%%%%%%%%%%%%%%%%%%%%

\section{Bridgeland moduli spaces}\label{sec:Bridgeland}

In this section, we study the moduli spaces arising in the higher-rank DT/PT correspondence in a particular case, building upon the recent work of Jardim, Lo, Maciocia, and Martinez \cite{JLMM2025HigherDTPT}.

Let $(X,H)$ be a polarized smooth, projective, connected threefold of Picard rank 1 over $\C$, where $H \subset X$ is a very ample divisor on $X$. Hence the real N\'eron--Severi group of $X$ is $\mathrm{N}^1(X)_\R \cong \R [H]$, where $[H]$ denotes the class of $H$ in $\mathrm{N}^1(X)_\R$. We also assume that $X$ is a threefold admitting Bridgeland stability conditions constructed via the double tilt procedure of \cite{bayer2014bridgeland}. See the introduction for a list of examples where such stability conditions are known to exist.

\subsection{Bridgeland stability conditions}

We first recall the double tilt construction of Bridgeland stability conditions on $X$, following \cite{bayer2014bridgeland}. For a numerical class $B \in \mathrm{N}^1(X)_\R$ and $E \in \Db(X)$, we denote by $\ch^B(E) \coloneqq \ch(E) \cdot e^{-B}$ the twisted Chern character of $E$, i.e.
\begin{align*}
    \ch_0^B(E) &= \ch_0(E), \quad \ch_1^B(E) = \ch_1(E) - B\ch_0(E),\\
    \ch_2^B(E) &= \ch_2(E) - B \ch_1(E) + \frac{B^2}{2}\ch_0(E),\\
    \ch_3^B(E) &= \ch_3(E) - B \ch_2(E) + \frac{B^2}{2}\ch_1(E) - \frac{B^3}{6}\ch_0(E).
\end{align*}
Let $\omega \in \Amp^1(X)_\R$ and consider the $B$-twisted slope function
\[
    \mu_{\omega, B}(E)= \begin{cases}
        \dfrac{\omega^2\ch_1^B(E)}{\ch_0^B(E)} \quad &\text{if } \ch_0(E) \neq 0,\\
        + \infty \quad &\text{if } \ch_0(E) = 0
    \end{cases}
\]
on $\Coh(X)$. It is known that this slope function defines a notion of semistability on $\Coh(X)$ which moreover admits the Harder-Narasimhan property. Therefore one obtains that the full subcategories
\begin{align*}
    \cT_{\omega, B} &= \{ E \in \Coh(X) \mid \mu_{\omega, B}(Q) > 0 \text{ for every quotient } E \twoheadrightarrow Q \} \\
    \cF_{\omega, B} &= \{ E \in \Coh(X) \mid \mu_{\omega, B}(F) \le 0 \text{ for every subsheaf } F \hookrightarrow E \}
\end{align*}
form a torsion pair on $\Coh(X)$. The extension-closed subcategory 
\[
    \cB_{\omega, B} \coloneqq \langle \cF_{\omega, B}[1],\cT_{\omega, B} \rangle \subseteq \Db(X)
\]
is the corresponding heart.

Now consider the slope function
\[
    \nu_{\omega, B}(E)= \begin{cases}
        \dfrac{\omega^2\ch_2^B(E) - \frac{\omega^3}{6}\ch_0^B(E)}{ \omega^2\ch_1^B(E)} \quad &\text{if } \omega^2\ch^B_1(E) \neq 0,\\
        + \infty \quad &\text{if } \omega^2\ch^B_1(E) = 0
    \end{cases}
\]
on $\cB_{\omega, B}$. As shown in \cite[Section 3]{bayer2014bridgeland}, the full subcategories
\begin{align*}
    \cT'_{\omega, B} &= \{ E \in \cB_{\omega, B} \mid \nu_{\omega, B}(Q) > 0 \text{ for every quotient } E \twoheadrightarrow Q \} \\
    \cF'_{\omega, B} &= \{ E \in \cB_{\omega, B} \mid \nu_{\omega, B}(F) \le 0 \text{ for every subsheaf } F \hookrightarrow E \}
\end{align*}
form a torsion pair on $\cB_{\omega, B}$, from which one obtains the heart
\[
    \cA_{\omega, B} \coloneqq \langle \cF'_{\omega, B}[1],\cT'_{\omega, B} \rangle \subseteq \Db(X).
\]

Next we place ourselves in the situation of \cite[p. 37]{JLMM2025HigherDTPT}. For every $(\beta, \alpha) \in \R \times \R_{> 0}$, set
\[
    \cA^{\beta,\alpha} \coloneqq \cA_{\alpha \sqrt{3}[H],\beta[H]},
\]
and for $s \in \R_{> 0}$ consider the central charge given by
\[
    Z_{\beta,\alpha,s}(E) \coloneqq - \ch_3^\beta(E) + \left(s + \frac{1}{6}\right)\alpha^2 \ch_1^\beta(E) + i\left(\ch_2^\beta(E) - \frac{\alpha^2}{2}\ch_0^\beta(E)\right),
\]
where $\ch_i^\beta$ is the real number satisfying the equality $\ch_i^\beta \cdot [H]^i = \ch_i^{\beta [H]}$ in the numerical group $\mathrm{N}^*(X)_\R$.

As in op. cit., we obtain a family $$(\cA^{\beta,\alpha}, Z_{\beta,\alpha,s}),\quad (\beta,\alpha,s) \in \R \times \R_{> 0} \times \R_{> 0}$$ of Bridgeland stability conditions (in the sense of \cite{bridgeland}). These stability conditions satisfy the support property and are moreover geometric, in the sense that the skyscraper sheaves $\cO_x$ are stable of the same phase. In this  case the notion of stability on $\cA^{\beta,\alpha}$ is defined with respect to the slope function
\[
    \lambda_{\beta,\alpha,s}(E) = \begin{cases}
        \dfrac{- \mathrm{Re}\, Z_{\beta,\alpha,s}(E)}{\im Z_{\beta,\alpha,s}(E)} \quad &\text{if } \im Z_{\beta,\alpha,s}(E) \neq 0,\\
        + \infty \quad &\text{if } \im Z_{\beta,\alpha,s}(E) = 0.
    \end{cases}
\]

To be consistent with the notation in \cref{sect:PT-semistable objects} (see \cref{rmk:convention}), we will consider the shifted Bridgeland stability conditions
\[
    \sigma_{\beta,\alpha,s} = (\cD^{\beta,\alpha}, -Z_{\beta,\alpha,s}), \quad \cD^{\beta,\alpha} \coloneqq \cA^{\beta,\alpha}[-1].
\]
Note that this choice does not change the slope function $\lambda_{\beta,\alpha,s}$, but only shifts the stable objects.

\begin{defn}
A nonzero object $E \in \cD^{\beta,\alpha}$ is $\lambda_{\beta,\alpha,s}$-\textbf{stable} (resp. $\lambda_{\beta,\alpha,s}$-\textbf{semistable}) if for any proper nonzero suboject $F \hookrightarrow E$ in $\cD^{\beta,\alpha}$
\[
    \lambda_{\beta,\alpha,s}(F) < \lambda_{\beta,\alpha,s}(E) \quad (\text{resp.}\, \le).
\]
We say that $E \in \cD^{\beta,\alpha}$ is $\lambda_{\beta,\alpha,s}$-\textbf{polystable} if
\[
    E \cong \bigoplus_i E_i,
\]
where each $E_i$ is $\lambda_{\beta,\alpha,s}$-stable and satisfies $\lambda_{\beta,\alpha,s}(E_i) = \lambda_{\beta,\alpha,s}(E)$.
\end{defn}

We recall that every $\lambda_{\beta,\alpha,s}$-semistable object $E \in \cD^{\beta,\alpha}$ admits a Jordan-H\"older filtration
\[
    0 = E_0 \subseteq E_1 \subseteq \ldots \subseteq E_m = E
\]
such that each factor $E_i/E_{i-1}$ is $\lambda_{\beta,\alpha,s}$-stable in $\cD^{\beta,\alpha}$ and $\lambda_{\beta,\alpha,s}(E_i/E_{i-1}) = \lambda_{\beta,\alpha,s}(E)$. Note that the corresponding graded module 
\[
    \gr(E) = \bigoplus_i E_i/E_{i-1}
\]
does not depend on the choice of the Jordan-H\"older filtration. We say that two semistable objects $E$ and $E'$ in $\cD^{\beta,\alpha}$ are \textbf{$S$-equivalent} if $\gr(E) \cong \gr(E')$.

\subsection{Moduli of semistable objects}
Fix $s > 0$ and choose a class $v \in \Knum(X)$ of positive rank. We denote by $\cM_{\beta,\alpha}(v) \subset \cD_{\mathrm{pug}}$ the moduli stack of $\lambda_{\beta,\alpha,s}$-semistable objects of class $v$ in $\cD^{\beta,\alpha}$, and by $\rM_{\beta,\alpha}(v)$ its corresponding good moduli space. The existence of the good moduli space is ensured by \cite[Example 7.29]{AHLH}. In this case $\rM_{\beta,\alpha}(v)$ is a proper algebraic space. 

We are interested in studying the Bridgeland moduli space $\rM_{\overline{\beta},\overline{\alpha}}(v)$ when $(\overline{\beta},\overline{\alpha})$ lies on the curve
\[
    \Theta_v \coloneqq \left\{ (\beta,\alpha) \in \R \times \R_{> 0} \mid \ch_2^\beta(v) = \frac{\alpha^2}{2}\ch_0^\beta(v) \right\}
\]
and $\overline{\beta}$ is sufficiently small as in \cite[Section 5]{JLMM2025HigherDTPT}. More precisely, if $(\beta_v,\alpha_v) \in \R \times \R_{> 0}$ is the intersection point between $\Theta_v$ and the largest numerical $\nu$-wall for $v$ to the left of the vertical line $$\{\, (\beta,\alpha) \in \R \times \R_{> 0} \mid \ch_1^\beta(v) = 0 \,\},$$ then we take $\overline{\beta} < \beta_v$. (See \cite[Section 4.1]{JardimMaciocia-Walls} for more details on the numerical $\nu$-walls in this context.)

If the numerical class $v$ is chosen so that $\rk(v)$ and $H^2\cdot \ch_1(v)$ are coprime, and $s > 1/3$ in the definition of the central charge $Z_{\beta,\alpha,s}$, then it was shown in \cite{JLMM2025HigherDTPT} that both the Gieseker moduli space $\rM^\mathrm{\mu}(v)$ of $\mu$-stable sheaves of class $v$ on $X$ and the moduli space $\rM^{\mathrm{PT}}_Z(v)$ of PT-stable objects can be realized as Bridgeland moduli spaces $\rM_{\beta,\alpha}(v)$ for suitable choices of $(\beta,\alpha)$. Moreover, these spaces are separated by a single wall, namely the curve $\Theta_v$, in the $(\beta,\alpha)$-plane. For any point $(\overline{\beta},\overline{\alpha}) \in \Theta_v$ with $\overline{\beta} < \beta_v$ there exists a wall-crossing diagram
\[
    \xymatrix{ \rM^\mathrm{\mu}(v) \ar[rd]^{\gamma} &  &  \rM^{\mathrm{PT}}_Z(v) \ar[ld]_\tau  \\  & \rM_{\overline{\beta},\overline{\alpha}}(v)  &  }
\]
where the morphisms are given by
\[
   \gamma(E) = [E^{[2]} \oplus (E^{[2]}/E)[-1]],\quad \tau(E) = [\cH^0(E) \oplus \cH^1(E)[-1]].
\]
Recall that $E^{[2]}$ is the $2$-closure of $E$, and so $E^{[2]}/E$ is $0$-dimensional; see \cref{def:closure}.

It is known that the Gieseker moduli space $\rM^\mu(v)$ is projective, and we have seen in \cref{sect:projPT} that the PT moduli space $\rM_Z^\mathrm{PT}(v)$ is also projective. In the next section, we complete the picture by showing that the Bridgeland moduli space $\rM_{\overline{\beta},\overline{\alpha}}(v)$ on the wall is projective as well.

\subsection{Projectivity of the Bridgeland moduli space}\label{subsect:projBridgeland}
In what follows, we do not need to assume that the rank and the degree of $v$ are coprime. Our goal is to show the projectivity of the Bridgeland moduli space $\rM_{\overline{\beta},\overline{\alpha}}(v)$ for $(\overline{\beta},\overline{\alpha}) \in \Theta_v$ and $\overline{\beta} < \beta_v$ as above.

We begin with a characterization of Bridgeland semistable objects, which is a restatement of \cite[Proposition 5.1]{JLMM2025HigherDTPT} (itself based on \cite[Propositions 3.2 and 5.2]{JardimMaciocia-Walls}). 

\begin{prop}\label{prop:pointsBMS}
Let $(\overline{\beta},\overline{\alpha}) \in \Theta_v$ with $\overline{\beta} < \beta_v$. Then, every object $E \in \Db(X)$ of class $v$ defines a $\C$-point in $\cM_{\overline{\beta},\overline{\alpha}}(v)$ if and only if $\cH^0(E)$ is $2$-Gieseker-semistable, $\cH^1(E)$ is $0$-dimensional, and $\cH^i(E) = 0$ for $i \neq 0,1$.

The $\C$-points of $\rM_{\overline{\beta},\overline{\alpha}}(v)$ correspond to $S$-equivalence classes of objects $E \in \cD^{\overline{\beta},\overline{\alpha}}$ of class $v$ and of the form $E = F \oplus T[-1]$, where $F$ is a $2$-Gieseker-semistable sheaf, and $T$ is a $0$-dimensional sheaf on $X$. 
\end{prop}

The above shows in fact that the Bridgeland moduli stack $\cM_{\overline{\beta},\overline{\alpha}}(v)$ coincides with the stack $\cM_X(v)$ introduced in \cref{sect:moduli spaces}.

\begin{cor}\label{cor:polystableBMS}
An object $E \in \cD^{\overline{\beta},\overline{\alpha}}$ of class $v$ is $\lambda_{\beta,\alpha,s}$-polystable  if and only if
\[  
    E \cong \left(\bigoplus_i F_i\right) \oplus \left(\bigoplus_j \cO_{p_j}[-1]\right)
\]
where each $F_i$ is a $2$-Gieseker-stable sheaf of slope $\lambda_{\overline{\beta},\overline{\alpha},s}(v) = +\infty$ and of homological dimension at most 1, and each $p_j \in X$ is a closed point. 
\end{cor}
\begin{proof}
We follow the proof of \cite[Proposition 5.1]{JLMM2025HigherDTPT}. By \cref{prop:pointsBMS}, if $E \in \cD^{\overline{\beta},\overline{\alpha}}$ is a $\lambda_{\overline{\beta},\overline{\alpha},s}$-semistable object of class $v$, then $\gr(E) \cong \gr(F \oplus T[-1])$, where
$F$ is a $2$-Gieseker-semistable sheaf, and $T$ is a $0$-dimensional sheaf. In this case both $F$ and $T[-1]$ lie in $\cD^{\overline{\beta},\overline{\alpha}}$ and satisfy
\[
    \lambda_{\overline{\beta},\overline{\alpha},s}(F) = \lambda_{\overline{\beta},\overline{\alpha},s}(T[-1]) = +\infty.
\]
Since we work with a \textit{geometric} Bridgeland stability condition, the Jordan-H\"older graded module associated to $T[-1]$ in $\cD^{\overline{\beta},\overline{\alpha}}$ satisfies
\[
    \gr(T[-1]) \cong \bigoplus_j \cO_{x_j}[-1]
\]
for some closed points $x_j \in X$. 

It remains to compute the graded factors corresponding to $F$. Let $F^{[2]}$ be the 2-closure of $F$ (see \cref{def:closure}). Then $Z_F \coloneqq F^{[2]}/F$ is 0-dimensional, and $F^{[2]}$ is a 2-Gieseker-semistable sheaf of homological dimension 1. The short exact sequence
\[
    0 \to F \to F^{[2]} \to Z_F \to 0
\]
in $\Coh(X)$ yields an exact triangle
\[
    Z_F[-1] \to F \to F^{[2]}
\]
in $\Db(X)$. Note that $Z_F[-1]$ and $F^{[2]}$ lie in $\cD^{\overline{\beta},\overline{\alpha}}$ and satisfy $$\lambda_{\overline{\beta},\overline{\alpha},s}(Z_F[-1]) = \lambda_{\overline{\beta},\overline{\alpha},s}(F) = \lambda_{\overline{\beta},\overline{\alpha},s}(F^{[2]}) = + \infty.$$ As before, we get
\[
    \gr(Z_F[-1]) \cong \bigoplus \cO_{y_k}[-1]
\]
for some closed points $y_k \in X$.

Now consider a Jordan-H\"older filtration of $F^{[2]}$ in $\Coh(X)$
\[
    0 = F_0 \subseteq F_1 \subseteq \ldots \subseteq F_m = F^{[2]}
\]
such that each factor $G_i \coloneqq F_i/F_{i-1}$ is torsion-free of homological dimension at most 1, $2$-Gieseker-stable and satisfying $\rp_{1}(G_i)=\rp_{1}(F^{[2]})$. Here $\rp_{1}$ denotes the reduced truncated Hilbert polynomial as defined in \cref{sect:prelim}. As shown in the proof of \cite[Proposition 5.1]{JLMM2025HigherDTPT}, we have that each $G_i$ is $\lambda_{\overline{\beta},\overline{\alpha},s}$-stable with $\lambda_{\overline{\beta},\overline{\alpha},s}(G_i) = +\infty$. Therefore
\[
    \gr(E) \cong \left(\bigoplus_i G_i\right) \oplus \left(\bigoplus_j \cO_{x_j}[-1] \right) \oplus \left(\bigoplus_k \cO_{y_k}[-1] \right)
\]
as claimed.
\end{proof}

Denote by $\cE$ the universal complex over $\cM_{\overline{\beta},\overline{\alpha}}(v) \times X$, and consider again the line bundle
\[
    \cL_{m,n,a} \coloneqq \lambda_\cE(w_{m,n,a} \cdot [\cO_D]) \in \Pic(\cM_{\overline{\beta},\overline{\alpha}}(v))
\]
where $w_{m,n,a}$ is as in \eqref{eq:grothClass} and $D \in |aH|$ is a divisor. As in \cref{L2descendstogms}, one obtains that $\cL_{m,n,a}$ descends to a line bundle $L_{m,n,a}$ on the good moduli space $\rM_{\overline{\beta},\overline{\alpha}}(v)$.

\begin{prop}\label{prop:globGenBMS}
There are integers $m,n,a > 0$ such that 
\begin{enumerate}[(a)]
    \item the line bundle $\cL_{m,n,a}$ is semiample over $\cM_{\overline{\beta},\overline{\alpha}}(v)$.
    \item for every morphism $S \to \cM_{\overline{\beta},\overline{\alpha}}(v)$ from a smooth, proper, connected curve $S$ with $\deg(\cL_{m,n,a}|_S = 0)$, if $E_1, E_2 \in \cM_{\overline{\beta},\overline{\alpha}}(v)$ are $\C$-points lying in the image of $S$, then $$\gr(\cH^0(E_1))^{[2]} \cong \gr(\cH^0(E_2))^{[2]}.$$
\end{enumerate}
\end{prop}
\begin{proof}
By \cref{prop:pointsBMS}, we see that the moduli stacks $\cM_{\overline{\beta},\overline{\alpha}}(v)$ and $\cM_X(v)$ coincide. Then the result follows from \cref{thm:SemiAmple} and \cref{prop:L1fiber}.
\end{proof}

Choose integers $m,n,a > 0$ as in \cref{prop:globGenBMS}. In particular, the line bundle $L_{m,n,a}$ is semiample on the good moduli space $\rM_{\overline{\beta},\overline{\alpha}}(v)$. Therefore for some $k > 0$ there exists a morphism
\[
    \varphi : \rM_{\overline{\beta},\overline{\alpha}}(v) \to \P\coloneqq \P \Gamma(\rM_{\overline{\beta},\overline{\alpha}}(v), L_{m,n,a}^{\otimes k})
\]
determined by the linear system $|L_{m,n,a}^{\otimes k}|$. Additionally, by choosing $k$ sufficiently large, we may assume that $\varphi$ has connected fibers. Fix a point $p \in \P$ and consider the commutative diagram
\[
    \begin{tikzcd}
        \cM_{\overline{\beta},\overline{\alpha}}(v) \ar[r] & \rM_{\overline{\beta},\overline{\alpha}}(v) \ar[r,"\varphi"] & \P \\
        \sM_{p} \ar[r,"\pi"] \ar[u,hookrightarrow] & \rM_{p} \ar[r] \ar[u,hookrightarrow] & \Spec \C \ar[u,"p"]
    \end{tikzcd}
\]
where $\rM_{p}$ and $\sM_{p}$ are the corresponding fibers of $\rM_{\overline{\beta},\overline{\alpha}}(v)$, resp. $\cM_{\overline{\beta},\overline{\alpha}}(v)$, over $p$. 

In what follows, for every $\C$-point $t \in \cM_{p}$, we denote as usual
\[
    E_t = \cE|^\LL_{\{t\} \times X}, \quad  F_t = \cH^0(E_t), \quad T_t = \cH^1(E_t).
\]

\begin{prop}\label{prop:sepFiberBMS}
There exists a $2$-Gieseker-semistable sheaf $F$ of homological dimension at most $1$ and of slope $\lambda_{\overline{\beta},\overline{\alpha},s}(F) = +\infty$ such that every $\C$-point $t \in \rM_{p}$ is $S$-equivalent to an object of the form
\[
    F \oplus \left(\bigoplus_{j \in J_t} \cO_{p_j}[-1]\right),
\]
where the $p_j \in X$ are closed points. 
\end{prop}
\begin{proof}
Let $S$ be a smooth, proper and connected curve, and consider a map $S \to \rM_p$. After taking a finite cover of $S$, we may assume by \cite[Lemma 5.4]{tajakka2022uhlenbeck} that the map lifts to $S \to \cM_{\overline{\beta},\overline{\alpha}}(v)$, and that we have $\deg(\cL_{m,n,a}|_S) = 0$. Now we can apply \cref{prop:globGenBMS} (b), which shows that the sheaves $\gr(F_t)^{[2]}$ are isomorphic as $t$ varies in the image of $S$. Since $\rM_p$ is a proper and connected algebraic space, it can be covered by smooth, proper and connected curves $S \to \rM_p$. Therefore, $\gr(F_t)^{[2]}$ does not vary with $t \in \rM_p$. Using \cref{cor:polystableBMS} we get the desired result.
\end{proof}

The zero-dimensional component of $\gr(E_t)$ is uniquely associated to an effective $0$-cycle $$C_t = \sum_{j \in J_t} n_j\langle p_j \rangle$$ on $X$. Note that $t \mapsto \deg(C_t)$ is a constant function on $\rM_{p}(\C)$, since the objects $E_t$ and $F$ have fixed Hilbert polynomials. We set $\ell \coloneqq \deg(C_t)$.

\begin{comment}
\begin{thm}
Assume that $\rk(v)$ and $H^2 \cdot \ch_1(v)$ are coprime. Then the line bundle $L_{m,n,a}$ is ample on $\rM_{\overline{\beta},\overline{\alpha}}(v)$, and therefore $\rM_{\overline{\beta},\overline{\alpha}}(v)$ is a projective scheme.
\end{thm}
\begin{proof}
Consider the proper map $\varphi : \rM_{\overline{\beta},\overline{\alpha}}(v) \to \P$ defined above, which satisfies $\varphi^*(\cO_\P(1)) \cong L_{m,n,a}^{\otimes k}$ for some power $k > 0$. By \cref{prop:sepFiberBMS}, we obtain that $\varphi$ has finite fibers. Indeed, this follows immediately since the length of each fiber is $l$ and its set-theoretic support does not vary over $\P$. Therefore $\varphi$ is in fact finite, cf. \cite[Tag 02LS]{stacks-project}, implying that the pullback $\varphi^*(\cO_\P(1)) \cong L_{m,n,a}^{\otimes k}$ of the ample line bundle $\cO_\P(1)$ on $\P$ is ample on $\rM_{\overline{\beta},\overline{\alpha}}(v)$. Hence $L_{m,n,a}$ is also ample.
\end{proof}
\end{comment}

\begin{prop}\label{prop:0Fibers}
The fiber $\rM_{p}$ over $p$ is $0$-dimensional. 
\end{prop}
\begin{proof}
Consider the moduli stack $\sN^\ell$ of $0$-dimensional sheaves of length $\ell$ on $X$, and denote by $\psi: \sN^\ell \to \rN^\ell$ its corresponding good moduli space. Recall that $\rN^\ell$ is a projective scheme, as it is isomorphic to the $\ell$-th symmetric product 
\[
    \mathrm{S}^\ell(X) \coloneqq \underbrace{(X \times \cdots \times X)}_{\ell \text{ times}}/\mathrm{S}_\ell
\]
where $\mathrm{S}_\ell$ is the symmetric group. Denote by $q$ the natural projection $\sN^\ell \times X \to X$, and let $\cT$ be the universal family of $0$-dimensional sheaves over $\sN^\ell \times X$. Then
\[
    (q^*F) \oplus \cT[-1] \in \Db(\sN^\ell \times X)
\]
is an $\sN^\ell$-perfect complex of Bridgeland semistable objects on $X$, which induces a natural commutative diagram
\[
    \begin{tikzcd}
        \sN^\ell \ar[r,"\psi"] \ar[d,"\gamma"] & \rN^\ell  \ar[d,"\varphi"] \\
        \sM_{p} \ar[r,"\pi"] & \rM_{p}
    \end{tikzcd}
\]
Clearly $\varphi$ is a finite surjective morphism, given the description of the geometric points of $\rM_{p}$ in \cref{prop:sepFiberBMS}. Furthermore, letting $\cE_{p}$ denote the restriction of the universal complex $\cE$ to $\cM_{p} \times X$, we have the following chain of isomorphisms
\begin{align*}
     \psi^*\varphi^*L_{m,n,a} \cong \gamma^*\pi^* L_{m,n,a} \cong{}& \gamma^*\lambda_{\cE_{p}}(w_{m,n,a}\cdot [\cO_D]) \\ 
     \cong{}& \lambda_{q^*F}(w_{m,n,a}\cdot [\cO_D]) \otimes \lambda_\cT(w_{m,n,a}\cdot [\cO_D])^{\vee} \\
     \cong{}& \lambda_\cT(w_{m,n,a}\cdot [\cO_D])^{\vee}\\
     \cong{}& \lambda_\cT(h)^{\otimes b}
\end{align*}
for some integer $b$. The last isomorphism follows since $w_{m,n,a}\cdot [\cO_D]$ is a linear combination of $h, h^2,h^3$ in $\rK(X)$, and $\lambda_\cT(h^{\otimes k})$ is trivial for $k > 1$ (as $\cT$ is a flat family of $0$-dimensional sheaves). 

We also know that $\lambda_\cT(h)$ descends to an ample line bundle on $\rN^\ell$; see \cite[Example 8.2.1]{HL} or \cite[Section 2.2.1]{pavel2021moduli} for details. Since $\varphi$ is finite and surjective, we obtain that $L_{m,n,a} \cong \cO_{\rM_{p}}$ is ample over $\rM_{p}$. This is possible only if $\rM_{p}$ is $0$-dimensional.
\end{proof}

\begin{thm}\label{thm:BMSproj}
There exist integers $m,n,a > 0$ such that $L_{m,n,a}$ is ample on $\rM_{\overline{\beta},\overline{\alpha}}(v)$. In particular, $\rM_{\overline{\beta},\overline{\alpha}}(v)$ is a projective scheme.
\end{thm}
\begin{proof}
We choose integers $m,n,a > 0$ as in \cref{prop:globGenBMS}, and consider the proper map $\varphi : \rM_{\overline{\beta},\overline{\alpha}}(v) \to \P$ defined above, which satisfies $\varphi^*(\cO_\P(1)) \cong L_{m,n,a}^{\otimes k}$ for some power $k > 0$. By \cref{prop:0Fibers}, we see that $\varphi$ has finite fibers, from which it follows by \cite[Proposition 3.1]{Olsson-Star} that $\rM_{\overline{\beta},\overline{\alpha}}(v)$ is a scheme.  Therefore $\varphi$ is in fact finite, cf. \cite[Tag 02LS]{stacks-project}, implying that the pullback $\varphi^*(\cO_\P(1)) \cong L_{m,n,a}^{\otimes k}$ of the ample line bundle $\cO_\P(1)$ on $\P$ is ample on $\rM_{\overline{\beta},\overline{\alpha}}(v)$. Hence $L_{m,n,a}$ is also ample, and $\rM_{\overline{\beta},\overline{\alpha}}(v)$ is projective.
\end{proof}

%%%%%%%%%%%%%%%%%%%%%%%%%%%%%%%%%%%%%%%%%%%%%%%%%%%%%
%%%%%%%%%%%%%%%%%%%%%%%%%%%%%%%%%%%%%%%%%%%%%%%%%%%%%
%%%%%%%%%%%%%%%%%%%%%%%%%%%%%%%%%%%%%%%%%%%%%%%%%%%%%

\bibliographystyle{alpha}
\addcontentsline{toc}{section}{References}
\bibliography{mainbib}

\medskip
\medskip
\begin{center}
\rule{0.4\textwidth}{0.4pt}
\end{center}
\medskip
\medskip

\end{document}